\documentclass[11pt,twoside]{amsart}
\usepackage{latexsym,amssymb,amsmath}

\numberwithin{equation}{section}
\def\demo{\noindent{\it Proof. }}
\newtheorem{theorem}{Theorem}[section]
\newtheorem{lemma}[theorem]{Lemma}
\newtheorem{proposition}[theorem]{Proposition}
\newtheorem{corollary}[theorem]{Corollary}
\newtheorem{conjecture}[theorem]{Conjecture}

\theoremstyle{definition}
\newtheorem{definition}[theorem]{Definition} 
 
\newtheorem{remark}[theorem]{Remark}
\newtheorem{example}[theorem]{Example}

\begin{document}


\title[Minimum distance functions and Reed-Muller-type codes]{Minimum
distance functions of graded ideals   
and Reed-Muller-type codes} 

\thanks{The first and third author were supported by SNI. The second author was
supported by CONACyT}
\thanks{*Corresponding author}
\author{Jos\'e Mart\'\i nez-Bernal}
\address{
Departamento de
Matem\'aticas\\
Centro de Investigaci\'on y de Estudios
Avanzados del
IPN\\
Apartado Postal
14--740 \\
07000 Mexico City, D.F.
}
\email{jmb@math.cinvestav.mx}

\author{Yuriko Pitones}
\address{
Departamento de
Matem\'aticas\\
Centro de Investigaci\'on y de Estudios
Avanzados del
IPN\\
Apartado Postal
14--740 \\
07000 Mexico City, D.F.
}
\email{ypitones@math.cinvestav.mx}

\author{Rafael H. Villarreal*}
\address{
Departamento de
Matem\'aticas\\
Centro de Investigaci\'on y de Estudios
Avanzados del
IPN\\
Apartado Postal
14--740 \\
07000 Mexico City, D.F.
}
\email{vila@math.cinvestav.mx}

\begin{abstract} 
We introduce and study the {\it minimum distance function\/} of a
graded ideal in a polynomial ring with coefficients in a field, and
show that it generalizes the minimum 
distance of projective Reed-Muller-type codes over finite fields. This gives an algebraic
formulation of the minimum distance of a projective Reed-Muller-type code 
in terms of the algebraic invariants and structure of the underlying
vanishing ideal. Then we give a
method, based on Gr\"obner bases and Hilbert functions, to find 
lower bounds for the minimum distance of
certain Reed-Muller-type codes. 
Finally we
show explicit upper bounds for the number of zeros of polynomials in 
a projective nested cartesian set and give some support to a conjecture of
Carvalho, Lopez-Neumann and L\'opez. 
\end{abstract}

\maketitle 

\section{Introduction}\label{intro-section}
Let $S=K[t_1,\ldots,t_s]=\oplus_{d=0}^{\infty} S_d$ be a polynomial ring over
a field $K$ with the standard grading and let $I\neq(0)$ be a graded ideal
of $S$ of Krull dimension $k$. The {\it Hilbert function} of $S/I$ is: 
$$
H_I(d):=\dim_K(S_d/I_d),\ \ \ d=0,1,2,\ldots,
$$
where $I_d=I\cap S_d$. By a theorem of Hilbert, there is a unique polynomial
$h_I(x)\in\mathbb{Q}[x]$ of 
degree $k-1$ such that $H_I(d)=h_I(d)$ for  $d\gg 0$. The
degree of the zero polynomial is $-1$.  

The {\it degree\/} or {\it multiplicity\/} of $S/I$ is the 
positive integer 
$$
\deg(S/I):=\left\{\begin{array}{ll}(k-1)!\, \lim_{d\rightarrow\infty}{H_I(d)}/{d^{k-1}}
&\mbox{if }k\geq 1,\\
\dim_K(S/I) &\mbox{if\ }k=0.
\end{array}\right.
$$ 

Let $\mathcal{F}_d$ be the set of 
all zero-divisors of $S/I$ not in $I$ of degree $d\geq 0$:
$$
\mathcal{F}_d:=\{\, f\in S_d\, \vert\, f\notin I,\, (I\colon f)\neq I\},
$$
where $(I\colon f)=\{h\in S\vert\, hf\in I\}$ is a quotient ideal.
Notice that $\mathcal{F}_0=\emptyset$.

The main object of study here is the function 
$\delta_I\colon \mathbb{N}\rightarrow \mathbb{Z}$ given by 
$$
\delta_I(d):=\left\{\begin{array}{ll}\deg(S/I)-\max\{\deg(S/(I,f))\vert\,
f\in\mathcal{F}_d\}&\mbox{if }\mathcal{F}_d\neq\emptyset,\\
\deg(S/I)&\mbox{if\ }\mathcal{F}_d=\emptyset.
\end{array}\right.
$$

We call $\delta_I$ the {\it minimum distance function\/} of
$I$. If $I$ is a prime ideal, then
$\mathcal{F}_d=\emptyset$ for all $d\geq 0$ and
$\delta_I(d)=\deg(S/I)$. We show that $\delta_I$ generalizes
the minimum distance function of projective Reed-Muller-type codes over finite
fields (Theorem~\ref{min-dis-vi}). 
This abstract algebraic formulation of the minimum distance gives a
new tool to study these type of linear codes.  

To compute $\delta_I(d)$ is a
difficult problem. For certain family of ideals 
we will give lower bounds for $\delta_I(d)$ which 
are easier to compute. 

Fix a monomial order $\prec$ on $S$. Let $\Delta_\prec(I)$  be the
{\it footprint\/} 
of $S/I$ consisting of all the {\it standard monomials\/} of $S/I$, with respect 
to $\prec$, and let $\mathcal{G}=\{g_1,\ldots,g_r\}$ be a Gr\"obner basis
of $I$. Then $\Delta_\prec(I)$ is the set of all monomials of $S$ that
are not a multiple of any of the leading monomials of
$g_1,\ldots,g_r$ (Lemma~\ref{nov6-15}). A polynomial $f$ is called {\it standard\/} if
$f\neq 0$ and $f$ is a
$K$-linear combination of standard monomials.

If $\Delta_\prec(I)\cap
S_d=\{t^{a_1},\ldots,t^{a_n}\}$ and $\mathcal{F}_{\prec,d}=\left.\left\{f=\sum_i{\lambda_i}t^{a_i}\, \right|\,
f\neq 0,\ \lambda_i\in
K,\, (I\colon f)\neq I\right\}$, then using the division algorithm
\cite[Theorem~3, p.~63]{CLO} we can write:
\begin{eqnarray*}
\delta_I(d)&=&
\deg(S/I)-\max\{\deg(S/(I,f))\vert\, f\in\mathcal{F}_d\}\\
&=&
\deg(S/I)-\max\{\deg(S/(I,f))\vert\, f\in\mathcal{F}_{\prec, d}\}.
\end{eqnarray*}

Notice that $\mathcal{F}_d\neq\emptyset$ if and only if
$\mathcal{F}_{\prec, d}\neq\emptyset$. 
If $K=\mathbb{F}_q$ is a finite field, then the number of standard
polynomials of degree $d$ is $n^q-1$, where $n$ is the
number of standard monomials of degree $d$.  
Hence, we can compute $\delta_I(d)$ for small values of
$n$ and $q$ (Examples~\ref{dec4-15} and \ref{example-min-dis}). 

Upper bounds for
$\delta_I(d)$ can be obtained by fixing a subset $\mathcal{F}_{\prec, d}'$ of
$\mathcal{F}_{\prec, d}$ and computing 
$$
\delta_I'(d)=
\deg(S/I)-\max\{\deg(S/(I,f))\vert\, f\in\mathcal{F}_{\prec, d}'\}\geq
\delta_I(d).
$$
Typically one uses $\mathcal{F}_{\prec,d}'=\left.\left\{f=\sum_i{\lambda_i}t^{a_i}\, \right|\,
f\neq 0,\ \lambda_i\in
\{0,1\},\, (I\colon f)\neq I\right\}$ or a subset of it.

Lower bounds for $\delta_I(d)$ are harder to find. Thus, we seek to
estimate $\delta_I(d)$ from below. So, with this in  
mind, we introduce the {\it footprint
function\/} 
of $I$: 
$$
{\rm fp}_I(d)=\left\{\begin{array}{ll}\deg(S/I)-\max\{\deg(S/({\rm
in}_\prec(I),t^a))\,\vert\,
t^a\in
\Delta_\prec(I)_d\}&\mbox{if }\Delta_\prec(I)_d\neq\emptyset,\\
\deg(S/I)&\mbox{if }\Delta_\prec(I)_d=\emptyset,
\end{array}\right.
$$
where ${\rm in}_\prec(I)=({\rm in}_\prec(g_1),\ldots,{\rm
in}_\prec(g_r))$ is the initial ideal of $I$, ${\rm
in}_\prec(g_i)$ is the initial monomial of $g_i$ for $i=1,\ldots,s$,
and $\Delta_\prec(I)_d=\Delta_\prec(I)\cap S_d$. 

The contents of this
paper are as follows. In Section~\ref{prelim-section} we 
present some of the results and terminology that will be needed
throughout the paper. 

Some of our results rely on 
a degree formula to compute the number of zeros that a homogeneous polynomial has
in any given finite set of points in a projective space 
(Lemma~\ref{degree-formula-for-the-number-of-zeros-proj}). 

In Section~\ref{mdf-section} we study
$\delta_I$ and present an alternative formula for $\delta_I$, pointed out to
us by Vasconcelos, valid for unmixed graded ideals (Theorem~\ref{wolmer-obs}). 
If $\mathcal{F}_d\neq\emptyset$ for $d\geq 1$ and
$I$ is unmixed, then, by
Lemma~\ref{degree-initial-footprint}, ${\rm fp}_I(d)$ is a lower bound of
$\delta_I(d)$, but even in this case ${\rm
fp}_I(d)$ could be negative (Example~\ref{footprint-bad}).  For this reason we will
sometimes make extra assumptions requiring that $\dim(S/I)\geq 1$ and
that $t_i$ is a zero-divisor of
$S/I$ for $i=1,\ldots,s$ (cf. Lemma~\ref{regular-init}).

One of our first main results gives some general properties of $\delta_I$
for an interesting class of graded ideals and show its relation to ${\rm
fp}_I$:

\noindent {\bf Theorem~\ref{footprint-lowerbound}.} 
{\it\ Let $I\subset S$ be an unmixed graded ideal of dimension
$\geq 1$ such that $t_i$ is a zero-divisor of $S/I$ for
$i=1,\ldots,s$. The following hold.
\begin{itemize}
\item[(i)] $\mathcal{F}_d\neq\emptyset$ for $d\geq 1$.
\item[(ii)] $\delta_I(d)\geq {\rm fp}_I(d)$ for $d\geq 1$.
\item[(iii)] $\deg(S/(I,t^a))\leq\deg(S/({\rm
in}_\prec(I),t^a))\leq\deg(S/I)$ for any
$t^a\in\Delta_\prec(I)\cap S_d$.
\item[(iv)] ${\rm fp}_I(d)\geq 0$.
\item[(v)] $\delta_I(d)\geq \delta_I(d+1)\geq 0$ 
for $d\geq 1$. 
\item[(vi)] If $I$ is a radical ideal and its associated primes are
generated by linear forms, then there is $r\geq 1$ such that 
$
\delta_I(1)>\cdots>\delta_I(r)=\delta_I(d)=1\
\mbox{ for }\ d\geq r.
$
\end{itemize}
}

The study of $\delta_I$ was motivated by the notion of 
minimum distance of linear codes in coding theory. For convenience we
recall this notion. Let $K=\mathbb{F}_q$ be a finite field. A {\it linear
code\/} is a linear subspace of $K^m$ for some 
$m$. The {\em basic parameters} of a linear code $C$ are {\it
length}:  $m$, 
{\it dimension}: $\dim_K(C)$, and {\it minimum distance\/}: 
$$
\delta(C):=\min\{\|v\|\colon 0\neq v\in C\},
$$
where $\|v\|$ is the number of non-zero
entries of $v$. 

The minimum distance of 
affine Reed-Muller-type codes has been studied using Gr\"obner bases techniques; see
\cite{carvalho,geil,geil-thomsen} and the references 
therein. Of particular interest to us is the {\it footprint
technique\/} introduced by Geil \cite{geil} to bound from below the minimum
distance. In this work we extend this technique to projective
Reed-Muller-type codes,  
a special type of linear codes that generalizes affine 
Reed-Muller-type codes \cite{affine-codes}. 
These projective codes are constructed
as follows. 

Let $K=\mathbb{F}_q$ be a finite field with $q$ elements,
let $\mathbb{P}^{s-1}$ be a projective space over 
$K$, and let $\mathbb{X}$ be a subset of
$\mathbb{P}^{s-1}$. The {\it vanishing ideal\/} of
$\mathbb{X}$, denoted $I(\mathbb{X})$,  is the ideal of $S$ 
generated by the homogeneous polynomials that vanish at all points of
$\mathbb{X}$. In this case the Hilbert function of $S/I(\mathbb{X})$ is denoted by
$H_\mathbb{X}(d)$. We can write
$\mathbb{X}=\{[P_1],\ldots,[P_m]\}\subset\mathbb{P}^{s-1}$ 
with $m=|\mathbb{X}|$.

Fix a degree $d\geq 0$. For each $i$ there is $f_i\in S_d$ such that
$f_i(P_i)\neq 0$. There is a $K$-linear map given by  
\begin{equation*}
{\rm ev}_d\colon S_d\rightarrow K^{m},\ \ \ \ \ 
f\mapsto
\left(\frac{f(P_1)}{f_1(P_1)},\ldots,\frac{f(P_m)}{f_m(P_m)}\right).
\end{equation*}

The image of $S_d$ under ${\rm ev}_d$, denoted by  $C_\mathbb{X}(d)$, is
called a {\it projective Reed-Muller-type code\/} of
degree $d$ on $\mathbb{X}$ \cite{duursma-renteria-tapia,GRT}. 
The {\it basic parameters} of the linear
code $C_\mathbb{X}(d)$ are:
\begin{itemize}
\item[(a)] {\it length\/}: $|\mathbb{X}|$,
\item[(b)] {\it dimension\/}: $\dim_K C_\mathbb{X}(d)$,
\item[(c)] {\it minimum distance\/}:
$\delta_\mathbb{X}(d):=\delta(C_\mathbb{X}(d))$. 
\end{itemize}

The {\it regularity\/} of $S/I(\mathbb{X})$, denoted 
${\rm reg}(S/I(\mathbb{X}))$, is the least integer $r\geq 0$ such that
$H_\mathbb{X}(d)$ is equal to $h_{I(\mathbb{X})}(d)$ for $d\geq r$.
As is seen below,  
the knowledge of the regularity of $S/I(\mathbb{X})$ is
important for applications to  coding theory. According to
\cite{geramita-cayley-bacharach} and 
Proposition~\ref{md-is-decreasing},  
there are integers $r\geq 0$ and $r_1\geq 0$ such that
$$
1=H_\mathbb{X}(0)<H_\mathbb{X}(1)<\cdots<H_\mathbb{X}(r-1)<H_\mathbb{X}(d)=|\mathbb{X}|$$
for $d\geq r= {\rm reg}(S/I(\mathbb{X}))$, and 
$$
|\mathbb{X}|=\delta_\mathbb{X}(0)>\delta_\mathbb{X}(1)>\cdots>\delta_\mathbb{X}(r_1-1)>\delta_\mathbb{X}(r_1)=\delta_\mathbb{X}(d)=1
\ \mbox{ for }\ d\geq r_1,
$$
respectively. The integer $r_1$ is called the 
{\it minimum distance regularity} of $S/I(\mathbb{X})$. In general
$r_1\leq r$ (see the discussion below). Using the methods of
\cite{algcodes,vanishing-ideals}, the regularity
of $S/I(\mathbb{X})$ can be effectively computed when $\mathbb{X}$ is
parameterized by monomials, but $r_1$ is very difficult to compute.      

The Hilbert function and the minimum distance are related by 
the Singleton bound: 
$$
1\leq \delta_\mathbb{X}(d)\leq |\mathbb{X}|-H_\mathbb{X}(d)+1.
$$
In particular, if $d\geq {\rm reg}(S/I(\mathbb{X}))\geq 1$, then
$\delta_\mathbb{X}(d)=1$. The converse is not true (Example~\ref{reg-min-dis}). Thus, potentially 
good Reed-Muller-type codes $C_\mathbb{X}(d)$ can occur only if $1\leq d < {\rm
reg}(S/I(\mathbb{X}))$. There are some families where $d\geq {\rm
reg}(S/I(\mathbb{X}))\geq 1$ 
if and only if $\delta_\mathbb{X}(d)=1$
\cite{cartesian-codes,ci-codes,sorensen}, but we do not know of any
set $\mathbb{X}$ parameterized by monomials where this fails. If
$\mathbb{X}$ is parameterized by monomials we say 
that $C_{\mathbb{X}}(d)$ is a {\it projective parameterized code\/}
\cite{algcodes,vanishing-ideals}.

A main problem in Reed-Muller-type codes is the following. 
If $\mathbb{X}$ has nice algebraic or combinatorial
structure, find formulas in terms of $s,q,d$, and the
structure of $\mathbb{X}$, for the {\em basic
parameters} of $C_\mathbb{X}(d)$: $H_\mathbb{X}(d)$, ${\rm
deg}(S/I(\mathbb{X}))$, $\delta_\mathbb{X}(d)$, and ${\rm
reg}(S/I(\mathbb{X}))$. Our main results can be used to study this
problem, especially when $\mathbb{X}$ is parameterized by monomials
or when $\mathbb{X}$ is a projective nested cartesian set (see
Definition~\ref{pncc}).  

The basic parameters of projective Reed-Muller-type codes have been
computed in some cases. If $\mathbb{X}=\mathbb{P}^{s-1}$, $C_\mathbb{X}(d)$ is the 
{\it classical projective
Reed--Muller code}. Formulas for its basic parameters were given in
\cite[Theorem~1]{sorensen}. If $\mathbb{X}$ is a projective torus (see
Definition~\ref{projectivetorus-def}), $C_\mathbb{X}(d)$ is
the {\it generalized projective Reed--Solomon code}. Formulas for its
basic parameters were given in \cite[Theorem~3.5]{ci-codes}. If $\mathbb{X}$ is the
image of a cartesian product of subsets of $K$, under the map
$K^{s-1}\rightarrow\mathbb{P}^{s-1}$, $x\rightarrow [x,1]$, then
$C_\mathbb{X}(d)$ is an {\it affine cartesian code} and formulas for its basic
parameters were given in \cite{geil-thomsen,cartesian-codes}.  

We give a formula, in terms of the degree, for the number of zeros 
in $\mathbb{X}\subset\mathbb{P}^{s-1}$ of any homogeneous polynomial
(Lemma~\ref{degree-formula-for-the-number-of-zeros-proj}). As a 
consequence we derive our second main result:

\noindent {\bf Theorem~\ref{min-dis-vi}.}{\it\ If $|\mathbb{X}|\geq 2$, then
$\delta_\mathbb{X}(d)=\delta_{I(\mathbb{X})}(d)\geq 1$ 
for $d\geq 1$.}

In particular, if $\prec$ is a monomial order on $S$, then
\begin{equation}\label{min-dis-vi-for}
\delta_\mathbb{X}(d)=\deg(S/I)-\max\{\deg(S/(I,f))\vert\,  
f\in\mathcal{F}_{\prec, d}\}.
\end{equation}

This description allows us to compute the minimum 
distance of Reed-Muller-type codes for small values of $q$
and $s$ (Corollary~\ref{oct15-15-1}) and it gives an algebraic
formulation of the minimum distance in terms of the algebraic
properties and invariants of the vanishing ideal. 
The formula of Eq.~(\ref{min-dis-vi-for}) is more interesting from 
the theoretical point of view than from a computational perspective.
Indeed, if $t_i$ is a zero-divisor of $S/I(\mathbb{X})$ for all $i$, 
then one has:
$$\delta_\mathbb{X}(d)\geq {\rm fp}_{I(\mathbb{X})}(d)\geq 0\ \mbox{
for }\ d\geq 1.$$

This inequality gives a lower bound for the minimum distance of any
Reed-Muller-type code 
over a set $\mathbb{X}$ parameterized by relatively prime monomials because in this case $t_i$ is
a zero-divisor of $S/I(\mathbb{X})$ for $i=1,\ldots,s$ 
(Corollary~\ref{footprint-lowerbound-vanishing}). 
Using SAGE \cite{sage} and a generator matrix of $C_\mathbb{X}(d)$ 
one can compute the minimum distance of $C_\mathbb{X}(d)$ 
in a more efficient way than by using our formula at least in the case
that $C_\mathbb{X}(d)$ arises from an affine Reed-Muller type code
\cite{affine-codes}.  

Let $d_1,\ldots,d_s$ be a non-decreasing sequence of positive integers
with $d_1\geq 2$ and $s\geq 2$, and let $L$ be the
ideal of $S$ generated by the set of all $t_it_j^{d_j}$ such that
$1\leq i<j\leq s$. It turns out that the ideal $L$ is
the initial ideal of the vanishing ideal of a 
projective nested cartesian set (see the discussion below). 
In Section~\ref{degree-formulas-section} we 
study the ideal $L$ and show some degree equalities 
as a preparation to show some applications.

Projective nested cartesian codes were introduced and studied in 
\cite{carvalho-lopez-lopez} (see Definition~\ref{pncc}). This type 
of evaluation codes generalize the classical projective Reed--Muller
codes \cite{sorensen}. As an application we will give some support for
the following interesting conjecture.

\noindent {{\bf Conjecture~\ref{carvalho-lopez-lopez-conjecture}.} {\rm
(Carvalho, Lopez-Neumann, and L\'opez \cite{carvalho-lopez-lopez})}} 
{\it\ Let $A_1,\ldots,A_s$ be subsets 
of $K$ and let $C_\mathcal{X}(d)$ be the $d$-th projective nested
cartesian code on the set
$\mathcal{X}=[A_1\times\cdots\times A_s]$ with $d_i=|A_i|$ for
$i=1,\ldots,s$. Then its minimum distance is given by 
$$
\delta_\mathcal{X}(d)=\left\{\hspace{-1mm}
\begin{array}{ll}\left(d_{k+2}-\ell+1\right)d_{k+3}\cdots d_s&\mbox{ if }
d\leq \sum\limits_{i=2}^{s}\left(d_i-1\right),\\
\qquad \qquad 1&\mbox{ if } d\geq
\sum\limits_{i=2}^{s}\left(d_i-1\right)+1,
\end{array}
\right.
$$
where $0\leq k\leq s-2$ and $\ell$ are integers such that 
$d=\sum_{i=2}^{k+1}\left(d_i-1\right)+\ell$ and $1\leq \ell \leq
d_{k+2}-1$.} 

Let $\prec$ be the lexicographical order on $S$ with 
$t_1\prec\cdots \prec t_s$. Carvalho et. al. 
found a Gr\"obner basis for $I(\mathcal{X})$ whose initial ideal
is $L$, and obtained formulas for the 
regularity and the degree of the coordinate ring $S/I(\mathcal{X})$
\cite{carvalho-lopez-lopez} (Proposition~\ref{carvalho-deg-reg}).

They showed the conjecture when the $A_i$'s are subfields of
$\mathbb{F}_q$, and essentially showed that their conjecture can be
reduced to:

\noindent {{\bf Conjecture~\ref{carvalho-lopez-lopez-conjecture-new}.} {\rm
(Carvalho, Lopez-Neumann, and L\'opez
\cite{carvalho-lopez-lopez})}}{\it\ 
If $f\in S_d$ is a standard polynomial such that
$(I(\mathcal{X}):f)\neq I(\mathcal{X})$,  
$1\leq d\leq\sum_{i=2}^{s}(d_i-1)$, and $V_\mathcal{X}(f)$ is the zero set
of $f$ in $\mathcal{X}$, then  
$$
|V_\mathcal{X}(f)|\leq \deg(
S/I(\mathcal{X}))-\left(d_{k+2}-\ell+1\right)d_{k+3}\cdots d_s,
$$
where $0\leq k\leq s-2$ and $\ell$ are integers such that 
$d=\sum_{i=2}^{k+1}\left(d_i-1\right)+\ell$ and $1\leq \ell \leq
d_{k+2}-1$. 
}

Let $f\neq 0$ be a standard polynomial, with respect to $\prec$, of
degree $d\geq 1$ with $d\leq \sum_{i=2}^{s}(d_i-1)$ and
$(I(\mathcal{X}):f)\neq I(\mathcal{X})$, 
and let  ${\rm in}_\prec(f)=t^a$ be its initial monomial. We can
write 
$$
t^a=t_r^{a_r}\cdots t_s^{a_s},
$$
with $1\leq r\leq s$,\,  $a_r\geq 1$,\,  $0\leq a_i\leq d_i-1$ for
$i>r$.  

We show an explicit upper bound for the number of zeros of $f$ in
$\mathcal{X}$:

\noindent {\bf Theorem~\ref{bounds-for-deg-init-f-1}.}{\it\
$|V_{\mathcal{X}}(f)|\leq 
\deg(S/({\rm
in}_\prec(I(\mathcal{X})),t^a))=$
$$
\left\{\begin{array}{ll}\displaystyle\deg
(S/I(\mathcal{X}))-\sum_{i=2}^{r+1}(d_i-a_i)\cdots(d_s-a_s)&\mbox{ if } 
a_r\leq d_r,\\
\displaystyle\deg(
S/I(\mathcal{X}))-(d_{r+1}-a_{r+1})\cdots(d_s-a_s)&\mbox{ if }
a_r\geq d_r+1.
\end{array}
\right.
$$
}

Then we use Theorem~\ref{bounds-for-deg-init-f-1} to give some support
for Conjecture~\ref{carvalho-lopez-lopez-conjecture-new}:

\noindent {\bf Theorem~\ref{yuriko-pepe-vila}.}{\it\ 
If $t_1$ divides  
$t^a={\rm in}_\prec(f)$, then  
$$
|V_\mathcal{X}(f)|\leq \deg
(S/I(\mathcal{X}))-\left(d_{k+2}-\ell+1\right)d_{k+3}\cdots d_s,
$$
where $0\leq k\leq s-2$ and $\ell$ are integers such that 
$d=\sum_{i=2}^{k+1}\left(d_i-1\right)+\ell$ and $1\leq \ell \leq
d_{k+2}-1$.}

As a consequence we show that the minimum distance of $C_\mathcal{X}(d)$ 
proposed in Conjecture~\ref{carvalho-lopez-lopez-conjecture} is in
fact the minimum distance of certain evaluation linear code
(Corollary~\ref{yuriko-pepe-vila-coro}). Finally in
Section~\ref{examples-section}, we show some examples that illustrate
how some of our results
can be used in practice.

\smallskip

For all unexplained
terminology and additional information,  we refer to 
\cite{BHer,CLO,Eisen} (for the theory of Gr\"obner bases, 
commutative algebra, and Hilbert functions), and
\cite{MacWilliams-Sloane,tsfasman} (for the theory of
error-correcting codes and linear codes). 

\section{Preliminaries}\label{prelim-section}

In this section, we 
present some of the results that will be needed throughout the paper
and introduce some more notation. All results of this
section are well-known. To avoid repetitions, we continue to employ
the notations and 
definitions used in Section~\ref{intro-section}.

Let $S=K[t_1,\ldots,t_s]=\oplus_{d=0}^{\infty}S_d$ be a graded
polynomial ring over a field
$K$, with the standard grading, and let $(0)\neq I\subset S$ be a
graded ideal. We will use the following 
multi-index notation: for $a=(a_1,\ldots,a_s)\in\mathbb{N}^s$, set
$t^a:=t_1^{a_1}\cdots t_s^{a_s}$. The multiplicative group of $K$ is
denoted by 
$K^*$. As usual, $\mathfrak{m}$ will denote the 
maximal ideal of $S$ generated by $t_1,\ldots,t_s$, and ${\rm ht}(I)$
will denote the height of the ideal $I$. By the {\rm dimension\/} 
of $I$ (resp. $S/I$) we mean the Krull dimension of $S/I$. The Krull dimension of
$S/I$ is denoted by $\dim(S/I)$.

One of the most useful and well-known facts about the degree is its additivity:

\begin{proposition}{\rm(Additivity of the degree
\cite[Proposition~2.5]{prim-dec-critical})}\label{additivity-of-the-degree}
If $I$ is an ideal of $S$ and 
$I=\mathfrak{q}_1\cap\cdots\cap\mathfrak{q}_m$ 
is an irredundant primary
decomposition, then
$$
\deg(S/I)=\sum_{{\rm ht}(\mathfrak{q}_i)={\rm
ht}(I)}\hspace{-3mm}\deg(S/\mathfrak{q}_i).$$
\end{proposition}

\begin{theorem}{\rm(Hilbert \cite[Theorem~4.1.3]{BHer})}\label{hilbert} 
Let $I\subset S$ be a graded ideal of dimension $k$. Then there is a 
polynomial $h_I(x)\in \mathbb{Q}[x]$ of degree $k-1$ 
such that $H_I(d)=h_I(d)$ for $d\gg 0$. 
\end{theorem}

If $f\in S$, the {\it quotient ideal\/} of $I$ with
respect to $f$ is given by $(I\colon f)=\{h\in S\vert\, hf\in I\}$.
The element $f$ is called a {\it zero-divisor\/} of $S/I$ if there is
$\overline{0}\neq \overline{a}\in S/I$ such that
$f\overline{a}=\overline{0}$, and $f$ is called {\it regular} on
$S/I$ if $f$ is not a zero-divisor. Notice that $f$ is a zero-divisor if
and only if $(I\colon f)\neq I$. An associated prime of $I$ is a prime
ideal $\mathfrak{p}$ of $S$ of the form $\mathfrak{p}=(I\colon f)$
for some $f$ in $S$.
 
\begin{theorem}{\cite[Lemma~2.1.19,
Corollary~2.1.30]{monalg-rev}}\label{zero-divisors} If $I$ is an
ideal of $S$ and   
$I=\mathfrak{q}_1\cap\cdots\cap\mathfrak{q}_m$ is 
an irredundant primary decomposition with ${\rm
rad}(\mathfrak{q}_i)=\mathfrak{p}_i$, then the set of zero-divisors
$\mathcal{Z}(S/I)$  of $S/I$ is given by
$$
\mathcal{Z}(S/I)=\bigcup_{i=1}^m\mathfrak{p}_i,
$$
and $\mathfrak{p}_1,\ldots,\mathfrak{p}_m$ are the associated primes of
$I$. 
\end{theorem}

In the introduction we defined the regularity of the coordinate
ring of a finite set in a projective space. This notion is defined 
for any graded ideal.

\begin{definition}
The {\it regularity\/} of $S/I$, denoted 
${\rm reg}(S/I)$, is the least integer $r\geq 0$ such that
$H_I(d)$ is equal to $h_I(d)$ for $d\geq r$.  
\end{definition}

If $I\subset S$ is Cohen-Macaulay 
and $\dim(S/I)=1$, then ${\rm reg}(S/I)$
is the Castelnuovo-Mumford regularity of $S/I$ in the sense of
\cite{eisenbud-syzygies}. If $\dim(S/I)=0$, then ${\rm reg}(S/I)$ is the least positive
integer $d$ such that $\mathfrak{m}^d\subset I$.

\paragraph{\bf The footprint of an ideal} Let  $\prec$ be a monomial
order on $S$ and let $(0)\neq I\subset S$ be an ideal. If $f$ is a non-zero 
polynomial in $S$, then one can write 
$$
f=\lambda_1t^{\alpha_1}+\cdots+\lambda_rt^{\alpha_r},
$$
with $\lambda_i\in K^*$ for all $i$ and 
$t^{\alpha_1}\succ\cdots\succ t^{\alpha_r}$. The {\it leading
monomial\/} $t^{\alpha_1}$ of $f$ 
is denoted by ${\rm in}_\prec(f)$. The {\it initial ideal\/} of $I$, denoted by
${\rm in}_\prec(I)$,  is the monomial ideal given by 
$${\rm in}_\prec(I)=(\{{\rm in}_\prec(f)|\, f\in I\}).
$$ 

A monomial $t^a$ is called a 
{\it standard monomial\/} of $S/I$, with respect 
to $\prec$, if $t^a$ is not the leading monomial of any polynomial in $I$.
The set of standard monomials, denoted $\Delta_\prec(I)$, is called the {\it
footprint\/} of $S/I$. The image of the standard polynomials of
degree $d$, under the canonical map $S\mapsto S/I$, 
$x\mapsto \overline{x}$, is equal to $S_d/I_d$, and the 
image of $\Delta_\prec(I)$ is a basis of $S/I$ as a $K$-vector space
(see \cite[Proposition~3.3.13]{monalg-rev}). In
particular, if $I$ is graded, then $H_I(d)$ is the number of standard
monomials of degree $d$. 

A subset $\mathcal{G}=\{g_1,\ldots, g_r\}$ of $I$ is called a 
{\it Gr\"obner basis\/} of $I$ if $${\rm
in}_\prec(I)=({\rm in}_\prec(g_1),\ldots,{\rm in}_\prec(g_r)).$$

\begin{lemma}{\cite[p.~2]{carvalho}}\label{nov6-15} Let $I\subset S$ be an ideal generated by
$\mathcal{G}=\{g_1,\ldots,g_r\}$, then
$$
\Delta_\prec(I)\subset\Delta_\prec({\rm
in}_\prec(g_1),\ldots,{\rm in}_\prec(g_r)),
$$
with equality if $\mathcal{G}$ is a Gr\"obner basis.
\end{lemma}

\begin{proof} Take $t^a$ in $\Delta_\prec(I)$. If
$t^a\notin\Delta_\prec({\rm
in}_\prec(g_1),\ldots,{\rm in}_\prec(g_r))$, then $t^a=t^c\,{\rm
in}_\prec(g_i)$ for some $i$ 
and some $t^c$. Thus $t^a={\rm
in}_\prec(t^c g_i)$, with $t^c g_i$ in $I$, a contradiction. 
The second statement holds by the definition of a Gr\"obner basis. 
\end{proof}

\begin{theorem}{\rm(Division algorithm
\cite[Theorem~3, p.~63]{CLO})} 
\label{division-algo}\index{division algorithm} 
If $f,g_1,\ldots,g_r$ are polynomials in $S$, then 
$f$ can be written as 
$$f=a_1g_1+\cdots+a_rg_r+h,$$
where 
$a_i,h\in S$ and either $h=0$ or $h\neq 0$ and no term of $h$ 
is divisible by one of the initial monomials\/ ${\rm
in}_\prec(g_1),\ldots,{\rm in}_\prec(g_r)$. 
Furthermore 
if $a_ig_i\neq 0$, then ${\rm in}_\prec(f)\geq{\rm in}_\prec(a_ig_i)$.
\end{theorem}

\begin{lemma}\label{regular-init} Let
$\mathcal{G}=\{g_1,\ldots,g_r\}$ be a Gr\"obner basis
of $I$. If for some $i$, the variable $t_i$ does not divides ${\rm
in}_\prec(g_j)$ for all $j$, then $t_i$ is a regular element on $S/I$.
\end{lemma}

\begin{proof} Assume that $t_if\in I$. By the division algorithm we
can write $f=g+h$, where $g\in I$ and $h$ is $0$ or a standard
polynomial. It suffices to show that $h=0$. If $h\neq 0$, then $t_i{\rm
in}_\prec(h)\in{\rm in}_\prec(I)$. Hence, using our hypothesis on
$t_i$, 
we get ${\rm
in}_\prec(h)\in{\rm in}_\prec(I)$, a contradiction. 
\end{proof}

This lemma tells us that if $t_i$ is a zero-divisor of $S/I$ for all
$i$, then any variable $t_i$ must occur in an  
initial monomial ${\rm in}_\prec(g_j)$ for some $j$.

\begin{theorem}{\rm(Macaulay \cite[Corollary~3.3.15]{monalg-rev})}\label{hilb=init}
If $I$ is a graded ideal of $S$, then $S/I$ and $S/{\rm in}_\prec(I)$ 
have the same Hilbert function and the same degree and regularity. 
\end{theorem}

\paragraph{\bf Vanishing ideals of finite sets} 
The {\it projective space\/} of 
dimension $s-1$
over the field $K$, denoted 
$\mathbb{P}_K^{s-1}$, or simply $\mathbb{P}^{s-1}$, is the quotient space 
$$(K^{s}\setminus\{0\})/\sim $$
where two points $\alpha$, $\beta$ in $K^{s}\setminus\{0\}$ 
are equivalent under $\sim$ if $\alpha=c\beta$ for some $c\in K^*$. It is usual to denote the 
equivalence class of $\alpha$ by $[\alpha]$. 

For a given a subset $\mathbb{X}\subset\mathbb{P}^{s-1}$ define 
$I(\mathbb{X})$, the {\it vanishing ideal\/} of $\mathbb{X}$, 
as the ideal generated by the homogeneous polynomials 
in $S$ that vanish at all points of $\mathbb{X}$, and 
given a graded ideal $I\subset S$ 
define its {\it zero set\/} relative to $\mathbb{X}$ as  
$$V_\mathbb{X}(I)=\left\{[\alpha]\in \mathbb{X}\vert\, 
f(\alpha)=0,\, 
\forall f\in I\, \mbox{ homogeneous} \right\}.
$$ 
In particular, if $f\in S$ is homogeneous, the zero set
$V_\mathbb{X}(f)$ of $f$ is the set of all $[\alpha]\in \mathbb{X}$
such that $f(\alpha)=0$, that is $V_\mathbb{X}(f)$ is the set of zeros
of $f$ in $\mathbb{X}$. 

\begin{lemma}\label{primdec-ix-a} Let $\mathbb{X}$ be a finite
subset of $\mathbb{P}^{s-1}$, let $[\alpha]$ be a point in
$\mathbb{X}$, 
with $\alpha=(\alpha_1,\ldots,\alpha_s)$
and $\alpha_k\neq 0$ for some $k$, and let
$I_{[\alpha]}$ be the vanishing ideal of $[\alpha]$. Then $I_{[\alpha]}$ is a prime ideal, 
\begin{equation*}
I_{[\alpha]}=(\{\alpha_kt_i-\alpha_it_k\vert\, k\neq i\in\{1,\ldots,s\}),\
\deg(S/I_{[\alpha]})=1,\,  
\end{equation*}
${\rm ht}(I_{[\alpha]})=s-1$, 
and $I(\mathbb{X})=\bigcap_{[\beta]\in{\mathbb{X}}}I_{[\beta]}$ is the primary
decomposition of $I(\mathbb{X})$. 
\end{lemma}

If $\mathbb{X}$ is a subset of $\mathbb{P}^{s-1}$ it is usual to
denote the Hilbert function of $S/I(\mathbb{X})$ by $H_\mathbb{X}$. 

\begin{proposition}{\rm\cite{geramita-cayley-bacharach}}\label{hilbert-function-dim=coro} 
If $\mathbb{X}\subset\mathbb{P}^{s-1}$ is a finite set, then 
$$
1=H_\mathbb{X}(0)<H_\mathbb{X}(1)<\cdots<H_\mathbb{X}(r-1)<H_\mathbb{X}(d)=|\mathbb{X}|$$
for $d\geq r= {\rm reg}(S/I(\mathbb{X}))$.
\end{proposition}

\begin{definition}\label{projectivetorus-def}\rm The set 
$\mathbb{T}=\{[(x_1,\ldots,x_s)]\in\mathbb{P}^{s-1}\vert\, x_i\in
K^*\, \forall\, i\}$ is called a {\it projective
torus\/}.
\end{definition}

\paragraph{\bf Projective Reed-Muller-type codes} 

Let $K=\mathbb{F}_q$ be a finite field, let $\mathbb{X}$ be a subset
of $\mathbb{P}^{s-1}$, and let $P_1,\ldots,P_m$ be a set of
representatives for the points of $\mathbb{X}$ with $m=|\mathbb{X}|$.
In this paragraph all results
are valid if we assume that $K$ is any field and that $\mathbb{X}$ is a finite subset of
$\mathbb{P}^{s-1}$ instead of assuming that $K$ is finite. However the
interesting case for coding theory is when $K$ is finite.

Fix a degree $d\geq 0$. 
For each $i$ there is $f_i\in S_d$ such that
$f_i(P_i)\neq 0$. Indeed suppose $P_i=[(a_1,\ldots,a_s)]$, there is at
least one $k$ in $\{1,\ldots,s\}$ such that $a_k\neq 0$. Setting
$f_i(t_1,\ldots,t_s)=t_k^d$ one has that $f_i\in S_d$ and
$f_i(P_i)\neq 0$. There is a $K$-linear map: 
\begin{equation*}
{\rm ev}_d\colon S_d=K[t_1,\ldots,t_s]_d\rightarrow K^{|\mathbb{X}|},\ \ \ \ \ 
f\mapsto
\left(\frac{f(P_1)}{f_1(P_1)},\ldots,\frac{f(P_m)}{f_m(P_m)}\right).
\end{equation*}

The map ${\rm ev}_d$ is called an {\it evaluation map}. The image of 
$S_d$ under ${\rm ev}_d$, denoted by  $C_\mathbb{X}(d)$, is called a {\it
projective Reed-Muller-type code\/} of degree $d$ over $\mathbb{X}$ 
\cite{duursma-renteria-tapia,GRT}. It is
also called an {\it evaluation code\/} associated to $\mathbb{X}$
\cite{gold-little-schenck}.

\begin{definition}\rm The {\it basic parameters} of the linear
code $C_\mathbb{X}(d)$ are its {\it length\/} $|\mathbb{X}|$, {\it
dimension\/} 
$\dim_K C_\mathbb{X}(d)$, and {\it minimum distance\/} 
$$
\delta_\mathbb{X}(d):=\min\{\|v\| 
\colon 0\neq v\in C_\mathbb{X}(d)\},
$$
where $\|v\|$ is the number of non-zero
entries of $v$. 
\end{definition}

\begin{lemma}\label{may21-15-1}
{\rm (a)} The map ${\rm ev}_d$ is well-defined, i.e., it is independent of
the set of representatives that we choose for the points of
$\mathbb{X}$. {\rm (b)} The basic parameters of the Reed-Muller-type code
$C_\mathbb{X}(d)$ are
independent of $f_1,\ldots,f_m$.
\end{lemma}

\begin{proof} (a): If $P_1',\ldots,P_m'$ is
another set of representatives, there are
$\lambda_1,\ldots,\lambda_m$ in $K^*$ such that $P_i'=\lambda_iP_i$
for all $i$. Thus, $f(P_i')/f_i(P_i')=f(P_i)/f_i(P_i)$ for $f\in S_d$
and $1\leq i\leq m$.

(b): Let $f_1',\ldots,f_m'$ be homogeneous polynomials of $S$
of degree $d$ such that $f_i'(P_i)\neq 0$ for $i=1,\ldots,m$, and let 
\begin{equation*}
{\rm ev}_d'\colon S_d\rightarrow K^{|\mathbb{X}|},\ \ \ \ \ 
f\mapsto
\left(\frac{f(P_1)}{f_1'(P_1)},\ldots,\frac{f(P_m)}{f_m'(P_m)}\right)
\end{equation*}
be the evaluation map relative to $f_1',\ldots,f_m'$. Then 
${\rm ker}({\rm ev}_d)={\rm ker}({\rm ev}_d')$ and  
$\|{\rm ev}_d(f)\|=\|{\rm ev}_d'(f)\|$ for $f\in S_d$. It follows
that the basic parameters of ${\rm ev}_d(S_d)$ and ${\rm
ev}_d'(S_d)$ are the same.
\end{proof}

\begin{lemma}\label{oct16-15} Let $\mathbb{Y}=\{[\alpha],[\beta]\}$ be a subset of
$\mathbb{P}^{s-1}$ with two elements. The following hold.
\begin{itemize}
\item[(i)] ${\rm reg}(S/I(\mathbb{Y}))=1$.
\item[(ii)] There is $h\in S_1$, a form of degree $1$, such that
$h(\alpha)\neq 0$ and $h(\beta)=0$. 
\item[(iii)] For each $d\geq 1$, there is $f\in S_d$, a form of
degree $d$, such that
$f(\alpha)\neq 0$ and $f(\beta)=0$. 
\item[(iv)] If $\mathbb{X}$ is a subset of $\mathbb{P}^{s-1}$
with at least two elements and $d\geq 1$, then there is $f\in S_d$
such $f\notin I(\mathbb{X})$ and $(I(\mathbb{X})\colon f)\neq
I(\mathbb{X})$. 
\end{itemize}
\end{lemma}
\begin{proof}
(i): As $H_\mathbb{Y}(0)=1$ and $|\mathbb{Y}|=2$, by
Proposition~\ref{hilbert-function-dim=coro}, we get that
$H_\mathbb{Y}(1)=|\mathbb{Y}|=2$. Thus $S/I(\mathbb{Y})$ has
regularity equal to $1$. 

(ii): Consider the evaluation map 
$$
{\rm ev}_1\colon S_1\longrightarrow K^2,\ \ \ \ \ 
f\mapsto \left(f(\alpha)/f_1(\alpha),f(\beta)/f_2(\beta)\right). 
$$
By part (i) this map is onto. Thus $(1,0)$ is
in the image of ${\rm ev}_1$ and the result follows. 

(iii): It follows from part (ii) by setting $f=h^d$.

(iv): By part (iii), there are distinct $[\alpha],[\beta]$ in
$\mathbb{X}$ and $f\in S_d$ such that $f(\alpha)\neq 0$, $f(\beta)=0$.
Then $f\notin I(\mathbb{X})$. 
Notice that $f(\beta)=0$ if and only if $f\in I_{[\beta]}$. 
Hence, by
Theorem~\ref{zero-divisors} and Lemma~\ref{primdec-ix-a}, $f$ is a
zero-divisor of $S/I(\mathbb{X})$, that is, 
$(I(\mathbb{X})\colon f)\neq I(\mathbb{X})$.
\end{proof}

The next result was shown in \cite[Proposition~5.2]{algcodes} and
\cite[Proposition~2.1]{tohaneanu} for some
special types of Reed-Muller-type codes (cf.
Theorem~\ref{footprint-lowerbound}(vi)). In \cite{algcodes} (resp.
\cite{tohaneanu}) it is assumed that $\mathbb{X}$ is contained in a
projective torus (resp. $\mathbb{X}$ is not contained in a hyperplane
and that there is $f\in S_d$ not vanishing at any point of
$\mathbb{X}$).

\begin{proposition}\label{md-is-decreasing} There is an integer
$r\geq 0$ such that
$$
|\mathbb{X}|=\delta_\mathbb{X}(0)>\delta_\mathbb{X}(1)>\cdots>\delta_\mathbb{X}(d)=\delta_\mathbb{X}(r)=1\
\mbox{ for }\ d\geq r.
$$
\end{proposition}

\begin{proof}
Assume that $\delta_\mathbb{X}(d)>1$, it suffices to show that
$\delta_\mathbb{X}(d)>\delta_\mathbb{X}(d+1)$. Pick $g\in S_d$ such
that $g\notin I(\mathbb{X})$ and 
$$
|V_\mathbb{X}(g)|=\max\{|V_{\mathbb{X}}(f)|\colon {\rm ev}_d(f)\neq 0;
f\in S_d\}.
$$
Then $\delta_\mathbb{X}(d)=|\mathbb{X}|-|V_\mathbb{X}(g)|\geq 2$. Thus
there are distinct points $[\alpha],[\beta]$ in $\mathbb{X}$ such that
$g(\alpha)\neq 0$ and $g(\beta)\neq 0$. By Lemma~\ref{oct16-15}, there
is a linear form $h\in S_1$ such that $h(\alpha)\neq 0$ and
$h(\beta)=0$. Hence the polynomial $hg$ is not in $I(\mathbb{X})$, has
degree $d+1$, and has at least $|V_\mathbb{X}(g)|+1$ zeros. 
Thus $\delta_\mathbb{X}(d)>\delta_\mathbb{X}(d+1)$, as required.
\end{proof}

The following summarizes the well-known relation between 
projective Reed-Muller-type codes and the theory of Hilbert functions.

\begin{proposition}{\rm(\cite{GRT}, \cite{algcodes})}\label{jan4-15}
The following hold. 
\begin{itemize}
\item[{\rm (i)}] $H_\mathbb{X}(d)=\dim_KC_\mathbb{X}(d)$ for $d\geq 0$.
\item[{\rm(ii)}] ${\rm deg}(S/I(\mathbb{X}))=|\mathbb{X}|$. 
\item[{\rm (iii)}] $\delta_\mathbb{X}(d)=1$ for $d\geq {\rm reg}(S/I(\mathbb{X}))$.  
\item[{\rm (iv)}] $S/I(\mathbb{X})$ is a Cohen--Macaulay reduced graded ring of dimension
$1$. 
\item[{\rm (v)}] $C_\mathbb{X}(d)\neq(0)$ for $d\geq 1$. 
\end{itemize}
\end{proposition}

\section{Computing the number of zeros using the
degree}\label{computing-zeros-section}
In this section we give a degree
formula to compute the number of 
zeros that a homogeneous polynomial has in any given finite set of
points in a projective space over any field.

An ideal $I\subset S$ is called {\it unmixed\/} 
if all its associated primes have the same height and $I$ is called
{\it radical\/} if $I$ is equal to its radical. The radical of $I$ is
denoted by ${\rm rad}(I)$.

\begin{lemma}\label{jul11-15} Let $I\subset S$ be a radical unmixed graded ideal. 
If $f\in S$ is homogeneous, $(I\colon f)\neq I$, and $\mathcal{A}$ is
the set of all associated primes
of $S/I$ that contain $f$, then ${\rm ht}(I)={\rm ht}(I,f)$ and 
$$
\deg(S/(I,f))=\sum_{\mathfrak{p}\in\mathcal{A}}\deg(S/\mathfrak{p}).
$$
\end{lemma}

\begin{proof} As $f$ is a zero-divisor of
$S/I$ and $I$ is 
unmixed, there is an associated prime ideal $\mathfrak{p}$ of $S/I$ of
height ${\rm ht}(I)$ such that
$f\in\mathfrak{p}$. Thus $I\subset(I,f)\subset\mathfrak{p}$, and
consequently  ${\rm ht}(I)={\rm ht}(I,f)$. Therefore the set of
associated primes of $(I,f)$ of height equal to ${\rm ht}(I)$ is not empty and is equal to 
 $\mathcal{A}$. There is an irredundant primary decomposition
\begin{equation}\label{jul10-15}
(I,f)=\mathfrak{q}_1\cap\cdots\cap\mathfrak{q}_r\cap\mathfrak{q}_{r+1}'\cap\cdots\cap\mathfrak{q}_t',
\end{equation}
where ${\rm rad}(\mathfrak{q}_i)=\mathfrak{p_i}$,
$\mathcal{A}=\{\mathfrak{p}_1,\ldots,\mathfrak{p}_r\}$, and ${\rm
ht}(\mathfrak{q}_i')>{\rm ht}(I)$ for $i>r$. We may assume that the
associated primes of $S/I$ are
$\mathfrak{p}_1,\ldots,\mathfrak{p}_m$. Since
$I$ is a radical ideal, we get that $I=\cap_{i=1}^m\mathfrak{p}_i$.
Next we show the following equality:
\begin{equation}\label{jul10-15-1}
\mathfrak{p}_1\cap\cdots\cap\mathfrak{p}_m=
\mathfrak{q}_1\cap\cdots\cap\mathfrak{q}_r\cap\mathfrak{q}_{r+1}'\cap\cdots\cap\mathfrak{q}_t'
\cap\mathfrak{p}_{r+1}\cap\cdots\cap\mathfrak{p}_m.
\end{equation}
The inclusion ``$\supset$'' is clear because
$\mathfrak{q}_i\subset\mathfrak{p}_i$ for $i=1,\ldots,r$. The 
inclusion ``$\subset$'' follows by noticing that the right hand side of
Eq.~(\ref{jul10-15-1}) is equal to 
$(I,f)\cap\mathfrak{p}_{r+1}\cap\cdots\cap\mathfrak{p}_m$, and 
consequently it contains $I=\cap_{i=1}^m\mathfrak{p}_i$. Notice that
${{\rm rad}}(\mathfrak{q}_j')=\mathfrak{p}_j'\not\subset\mathfrak{p}_i$ for
all $i,j$ and $\mathfrak{p}_j\not\subset\mathfrak{p}_i$ for $i\neq j$.
Hence localizing Eq.~(\ref{jul10-15-1}) 
at the prime ideal $\mathfrak{p}_i$ for $i=1,\ldots,r$, 
we get that $\mathfrak{p}_i=I_{\mathfrak{p}_i}\cap
S=(\mathfrak{q}_i)_{\mathfrak{p}_i}\cap S=\mathfrak{q}_i$ for
$i=1,\ldots,r$. Using Eq.~(\ref{jul10-15}) and the additivity of the
degree the required equality follows.
\end{proof}

\begin{lemma}\label{degree-formula-for-the-number-of-zeros-proj}
Let $\mathbb{X}$ be a finite subset of 
$\mathbb{P}^{s-1}$ over a field $K$ and let $I(\mathbb{X})\subset S$ be its
graded vanishing ideal. If 
$0\neq f\in S$ is homogeneous, then the number of zeros of $f$ in
$\mathbb{X}$ is given by 
$$
|V_{\mathbb{X}}(f)|=\left\{
\begin{array}{cl}
\deg S/(I(\mathbb{X}),f)&\mbox{if }(I(\mathbb{X})\colon f)\neq
I(\mathbb{X}),\\ 
0&\mbox{if }(I(\mathbb{X})\colon f)=I(\mathbb{X}).
\end{array}
\right.
$$
\end{lemma}

\begin{proof} Let $[P_1],\ldots,[P_m]$ be the points of $\mathbb{X}$
with $m=|\mathbb{X}|$, and let $[P]$ be a point in
$\mathbb{X}$, with $P=(\alpha_1,\ldots,\alpha_s)$
and $\alpha_k\neq 0$ for some $k$. Then the vanishing ideal $I_{[P]}$
of $[P]$ is a prime ideal of height $s-1$, 
\begin{equation*}
I_{[P]}=(\{\alpha_kt_i-\alpha_it_k\vert\, k\neq i\in\{1,\ldots,s\}),\
\deg(S/I_{[P]})=1,
\end{equation*}
and $I(\mathbb{X})=\bigcap_{i=1}^mI_{[P_i]}$ is a primary
decomposition (see Lemma~\ref{primdec-ix-a}). In particular 
$I(\mathbb{X})$ is an unmixed radical ideal of dimension $1$. 

Assume that $(I(\mathbb{X})\colon f)\neq I(\mathbb{X})$. Let
$\mathcal{A}$ be the set of all $I_{[P_i]}$ that contain the
polynomial $f$. Then 
$f(P_i)=0$ if and only if $I_{[P_i]}$ is in $\mathcal{A}$. Hence, 
by Lemma~\ref{jul11-15}, we get 
$$
|V_\mathbb{X}(f)|=\sum_{[P_i]\in
V_\mathbb{X}(f)}\deg S/I_{[P_i]}=\sum_{f\in I_{[P_i]}}\deg S/I_{[P_i]}=\deg S/(I(\mathbb{X}),f).
$$
If $(I(\mathbb{X})\colon f)=I(\mathbb{X})$, then $f$ is
a regular element of $S/I(\mathbb{X})$. This means that $f$ is not in
any of the associated primes of $I(\mathbb{X})$, that is, $f\notin
I_{[P_i]}$ for all $i$. Thus $V_\mathbb{X}(f)=\emptyset$ and
$|V_\mathbb{X}(f)|=0$.
\end{proof}

\section{The minimum distance function of a graded
ideal}\label{mdf-section}

In this section we study the minimum distance function $\delta_I$ of a graded
ideal $I$ and show that it generalizes the
minimum distance function of a projective Reed-Muller-type code. To
avoid repetitions, we continue to employ 
the notations and definitions used in Sections~\ref{intro-section} and
\ref{prelim-section}. 

The next result will be used to bound the number of zeros
of polynomials over finite fields (see
Corollary~\ref{poly-bounds-initial}) 
and to study the general properties of $\delta_I$. 

\begin{lemma}\label{degree-initial-footprint} Let $I\subset S$ be an unmixed graded ideal and let $\prec$ be
a monomial order. If $f\in S$ is homogeneous and $(I\colon f)\neq I$, then
$$
\deg(S/(I,f))\leq\deg(S/({\rm
in}_\prec(I),{\rm in}_\prec(f)))\leq\deg(S/I),
$$
and $\deg(S/(I,f))<\deg(S/I)$ if $I$ is an unmixed radical ideal and $f\notin I$.
\end{lemma}

\begin{proof} To simplify notation we set $J=(I,f)$ and 
$L=({\rm in}_\prec(I),{\rm in}_\prec(f))$. We denote the Krull
dimension of $S/I$ by $\dim(S/I)$. 
Recall that $\dim(S/I)=\dim(S)-{\rm ht}(I)$. First we show that $S/J$ and
$S/L$ have Krull dimension equal to $\dim(S/I)$. As $f$ is a zero-divisor of
$S/I$ and $I$ is 
unmixed, there is an
associated prime ideal $\mathfrak{p}$ of $S/I$ such that
$f\in\mathfrak{p}$ and $\dim(S/I)=\dim(S/\mathfrak{p})$. Since 
$I\subset J\subset\mathfrak{p}$, we get that $\dim(S/J)$ is 
$\dim(S/I)$.  
Since $S/I$ and $S/{\rm
in}_\prec(I)$ have the same Hilbert function, and so does
$S/\mathfrak{p}$ and $S/{\rm in}_\prec(\mathfrak{p})$, we obtain
$$
\dim(S/{\rm in}_\prec(I))=\dim(S/I)=\dim(S/\mathfrak{p})=\dim(S/{\rm
in}_\prec(\mathfrak{p})).
$$
Hence, taking heights in the inclusions ${\rm in}_\prec(I)\subset L\subset{\rm
in}_\prec(\mathfrak{p})$, we obtain ${\rm ht}(I)={\rm ht}(L)$. 

Pick a Gr\"obner basis $\mathcal{G}=\{g_1,\ldots,g_r\}$ of $I$. Then
$J$ is generated by $\mathcal{G}\cup\{f\}$ and by Lemma~\ref{nov6-15}
one has the inclusions
\begin{eqnarray*}
& &\Delta_\prec(J)=\Delta_\prec(I,f)\subset\Delta_\prec({\rm in}_\prec(g_1),\ldots,{\rm
in}_\prec(g_r),{\rm in}_\prec(f))=\\
& &\ \ \ \ \  \Delta_\prec({\rm in}_\prec(I),{\rm
in}_\prec(f))=\Delta_\prec(L)\subset \Delta_\prec({\rm in}_\prec(g_1),\ldots,{\rm
in}_\prec(g_r))=\Delta_\prec(I).
\end{eqnarray*}
Thus $\Delta_\prec(J)\subset \Delta_\prec(L)\subset \Delta_\prec(I)$.
Recall that $H_I(d)$, the Hilbert function of $I$ at $d$, is the number of standard
monomials of degree $d$. Hence $H_J(d)\leq H_L(d)\leq H_I(d)$ for
$d\geq 0$. If $\dim(S/I)$ is equal to $0$, then 
$$
\deg(S/J)=\sum_{d\geq 0}H_J(d)\leq \deg(S/L)=\sum_{d\geq 0}H_L(d)\leq
\deg(S/I)=\sum_{d\geq 0}H_I(d).
$$

Assume now that $\dim(S/I)\geq 1$. 
By the Hilbert theorem, $H_J$, $H_L$, $H_I$ are polynomial functions of degree
equal to $k=\dim(S/I)-1$. Thus 
$$k!\lim_{d\rightarrow\infty} H_J(d)/d^k\leq
k!\lim_{d\rightarrow\infty} H_L(d)/d^k\leq
k!\lim_{d\rightarrow\infty} H_I(d)/d^k,$$
that is $\deg(S/J)\leq\deg(S/L)\leq\deg(S/I)$. 

If $I$ is an unmixed radical ideal and $f\notin I$, then there is at
least one minimal prime that does not contains $f$. Hence, by
Lemma~\ref{jul11-15}, it follows that $\deg(S/(I,f))<\deg(S/I)$.
\end{proof}

\begin{remark}\rm Let $I\subset S$ be an unmixed graded ideal of
dimension $1$. If $f\in S_d$, then
$(I\colon f)=I$ if and only if $\dim(S/(I,f))=0$. In this case 
 $\deg(S/(I,f))$ could be greater than $\deg(S/I)$.
\end{remark}

\begin{corollary}\label{poly-bounds-initial} Let $\mathbb{X}$ be a
finite subset of $\mathbb{P}^{s-1}$, let $I(\mathbb{X})\subset
S$ be its vanishing ideal, and let $\prec$ be a monomial order. If 
$0\neq f\in S$ is homogeneous and $(I(\mathbb{X})\colon f)\neq I(\mathbb{X})$, then
$$
|V_{\mathbb{X}}(f)|=\deg(S/(I(\mathbb{X}),f))\leq\deg(S/({\rm
in}_\prec(I(\mathbb{X})),{\rm in}_\prec(f)))\leq\deg(S/I(\mathbb{X})),
$$
and $\deg(S/(I(\mathbb{X}),f))<\deg(S/I(\mathbb{X}))$ if $f\notin
I(\mathbb{X})$.
\end{corollary}

\begin{proof} It follows from
Lemma~\ref{degree-formula-for-the-number-of-zeros-proj} 
and Lemma~\ref{degree-initial-footprint}. 
\end{proof}

The next alternative formula for the minimum distance function is
valid for unmixed graded ideals. It was pointed out to
us by Vasconcelos.

\begin{theorem}\label{wolmer-obs}
Let $I\subset S$ be an unmixed graded ideal and let $\prec$ be a monomial
order on $S$. If $\Delta_\prec(I)_d^p$ 
 is the set of homogeneous
standard polynomials of degree $d$ and $\mathfrak{m}^d\not\subset I$, then
\begin{eqnarray*}
\delta_I(d)&=&\min\{\deg(S/(I\colon f))\,\vert\, f\in S_d\setminus
I\}\\
&=&\min\{\deg(S/(I\colon f))\,\vert\, f\in \Delta_\prec(I)_d^p\}.
\end{eqnarray*}
\end{theorem}

\begin{proof} The second equality is clear because by the division
algorithm any $f\in
S_d\setminus I$ can be written as $f=g+h$, where $g\in I$ and 
$h\in  \Delta_\prec(I)_d^p$, and $(I\colon f)=(I\colon h)$. Next we
show the first equality. If $\mathcal{F}_d=\emptyset$,
$\delta_I(d)=\deg(S/I)$ and for any $f\in S_d\setminus I$, one has that
$(I\colon f)$ is equal to $I$. Thus equality holds.
Assume that $\mathcal{F}_d\neq\emptyset$. Take $f\in \mathcal{F}_d$.
Using that $I$ is unmixed, it is not hard to see that 
$S/I$, $S/(I\colon f)$, and $S/(I,f)$ have the same dimension.
There are exact sequences
\begin{eqnarray*}
& &0\longrightarrow (I\colon f)/I 
\longrightarrow S/I\longrightarrow S/(I\colon f)
\longrightarrow 0,\mbox{ and }\\ 
& &\ \ \ \  \ \  0\longrightarrow
(I\colon f)/I \longrightarrow(S/I)[-d]\stackrel{f}{\longrightarrow}
S/I\longrightarrow
S/(I,f)\longrightarrow 0.
\end{eqnarray*}
Hence, by the additivity of Hilbert functions, we get 
\begin{equation}\label{dec1-15}
H_I(i)-H_{(I\colon f)}(i)=H_I{(i-d)}-H_I(i)+H_{(I,f)}(i)\ \mbox{
for }i\geq 0.
\end{equation}

By definition of $\delta_I(d)$ it suffices to show 
the following equality 
\begin{equation}\label{dec8-15} 
\deg(S/(I\colon f))=\deg(S/I)- \deg(S/(I, f)).
\end{equation}

If $\dim S/I=0$, then using Eq.~(\ref{dec1-15}) one has
$$
\sum_{i\geq 0}H_I(i)-\sum_{i\geq 0}H_{(I\colon f)}(i)=\sum_{i\geq
0}H_I{(i-d)}-\sum_{i\geq 0}H_I(i)+\sum_{i\geq 0}H_{(I,f)}(i).
$$

Hence, using the definition of degree, the
equality of Eq.~(\ref{dec8-15}) follows. If $k=\dim S/I-1\geq 0$, by
the Hilbert theorem, $H_{I}$, $H_{(I,f)}$, and $H_{(I\colon f)}$ are polynomial
functions of degree $k$. Then dividing 
Eq.~(\ref{dec1-15}) by $i^{k}$ and taking limits as $i$ goes to
infinity, the equality of Eq.~(\ref{dec8-15}) holds.
\end{proof}

We come to one of our main results.

\begin{theorem}\label{footprint-lowerbound}
Let $\prec$ be a monomial order and let $I\subset S$ be an unmixed ideal of dimension
$\geq 1$ such that $t_i$ is a zero-divisor of $S/I$ for
$i=1,\ldots,s$. The following hold. 
\begin{itemize}
\item[(i)] The set $\mathcal{F}_d=\{f\in S_d\colon
f\notin I,\, (I\colon f)\neq I\}$ is not empty for $d\geq 1$.
\item[(ii)] $\delta_I(d)\geq {\rm fp}_I(d)$ for $d\geq 1$.
\item[(iii)] $\deg S/(I,t^a)\leq\deg S/({\rm
in}_\prec(I),t^a)\leq\deg S/I$ for any
$t^a\in\Delta_\prec(I)\cap S_d$.
\item[(iv)] ${\rm fp}_I(d)\geq 0$.
\item[(v)] $\delta_I(d)\geq \delta_I(d+1)\geq 0$ 
for $d\geq 1$. 
\item[(vi)] If $I$ is a radical ideal and its associated primes 
are generated by linear forms, then there is an integer $r\geq 1$ such that 
$$
\delta_I(1)>\cdots>\delta_I(r)=\delta_I(d)=1\
\mbox{ for }\ d\geq r.
$$
\end{itemize}
\end{theorem}

\begin{proof} (i): Since $\dim(S/I)\geq 1$, there is $1\leq \ell\leq s$ such that
$t_\ell^d$ is not in $I$, and $(I\colon t_\ell^d)\neq I$ because
$t_\ell^d$ is a zero-divisor of $S/I$. Thus $t_\ell^{d}$ is in
$\mathcal{F}_d$.

(ii): The set $\mathcal{F}_d$ is not empty for $d\geq 1$ by
part (i). Pick a polynomial $F$ of degree $d$ such that 
$\delta_I(d)=\deg(S/I)-\deg(S/(I,F))$, $F\notin I$, and $(I\colon
F)\neq I$. We may assume that $F$ is a sum of standard monomials of
$S/I$ with respect to $\prec$ (this follows using a 
Gr\"obner basis of $I$ and the division algorithm). Then, by
Lemma~\ref{degree-initial-footprint}, we get 
$$
\deg S/(I,F)\leq\deg S/({\rm
in}_\prec(I),{\rm in}_\prec(F))\leq\deg S/I,
$$
where ${\rm in}_\prec(F)$ is a standard monomial of $S/I$. Therefore 
$\delta_I(d)\geq {\rm fp}_I(d)$ for $d\geq 1$.

(iii), (iv): Since any standard monomial of degree $d$ is a
zero-divisor, by Lemma~\ref{degree-initial-footprint}, we get 
the inequalities in item (iii). Part (iv) follows at once from part
(iii).

(v): The set $\mathcal{F}_d$ is not empty for $d\geq 1$ by part (i). 
Then, by parts (ii) and (iv), $\delta_I(d)\geq 0$. Pick $F\in S_d$ such
that $F\notin I$, $(I\colon F)\neq I$ and 
$$
\deg(S/(I,F))=\max\{\deg(S/(I,f))\vert\,
f\notin I,\, f\in S_d,\, (I\colon f)\neq I\}.
$$

There is $h\in S_1$ such that $hF\notin I$, because otherwise 
$\mathfrak{m}=(t_1,\ldots,t_s)$ is an associated prime of $S/I$,
a contradiction to the assumption that $I$ is unmixed of dimension
$\geq 1$. As $F$ is a zero-divisor of $S/I$, so is $hF$. The ideals
$(I,F)$ and $(I,hF)$ have height equal to ${\rm ht}(I)$. Therefore taking Hilbert
functions in the exact sequence 
$$
0\longrightarrow (I,F)/(I,hF)\longrightarrow S/(I,hF)\longrightarrow
S/(I,F)\longrightarrow 0
$$
it follows that $\deg(S/(I,hF))\geq \deg(S/(I,F))$. This proves that
$\delta_I(d)\geq \delta_I(d+1)$. 

(vi): By Lemma~\ref{degree-initial-footprint}, $\delta_I(d)\geq 1$
for $d\geq 1$. Assume that $\delta_I(d)>1$. By part (v) it suffices to
show that $\delta_I(d)>\delta_I(d+1)$. Pick a polynomial $F$ as in
part (v). Let $\mathfrak{p}_1,\ldots,\mathfrak{p}_m$ be the associated
primes of $I$. Then, by Lemma~\ref{jul11-15}, one has
\begin{eqnarray*}
\delta_I(d)&=&\deg(S/I)-\deg(S/(I,F))\\ 
&=&\sum_{i=1}^m\deg(
S/\mathfrak{p}_i)-\sum_{F\in\mathfrak{p}_i}\deg(S/\mathfrak{p}_i)\geq 2.
\end{eqnarray*}
Hence there are $\mathfrak{p}_k\neq\mathfrak{p}_j$ such that $F$ is
not in $\mathfrak{p}_k\cup \mathfrak{p}_j$. Pick a linear form $h$
in $\mathfrak{p}_k\setminus\mathfrak{p}_j$; which exists because $I$
is unmixed and $\mathfrak{p}_k$ is generated by linear forms. Then $h F\notin I$
because $h F\notin\mathfrak{p}_j$, and $h F$ is a
zero-divisor of $S/I$ because $(I\colon F)\neq I$. Noticing that 
$F\notin\mathfrak{p}_k$ and $h F\in\mathfrak{p}_k$, by
Lemma~\ref{jul11-15}, we get
$$
\deg(S/(I,F))=\sum_{F\in\mathfrak{p}_i}\deg(S/\mathfrak{p}_i)<
\sum_{h F\in\mathfrak{p}_i}\deg(S/\mathfrak{p}_i)=\deg
(S/(I,h F)).
$$
Therefore $\delta_I(d)>\delta_I(d+1)$.
\end{proof}

\begin{corollary} If  $I\subset S$ is a Cohen-Macaulay square-free
monomial ideal, then there is an integer $r\geq 1$ such that 
$$
\delta_I(1)>\cdots>\delta_I(r)=\delta_I(d)=1\
\mbox{ for }\ d\geq r.
$$
\end{corollary}

\begin{proof} If $I$ is prime, then $I$ is generated by a subset of
$\{t_1,\ldots,t_s\}$, $\deg(S/I)=1$, and $\mathcal{F}_d=\emptyset$ for
all $d$. Hence $\delta_I(d)=1$ for $d\geq 1$. Thus we may assume that $I$ has at least
two associated primes. Any Cohen-Macaulay ideal is unmixed
\cite{monalg-rev}. Thus the degree of $S/I$ is the number of associated
primes of $I$. Hence, we may assume that all variables are
zero-divisors of $S/I$ and the result follows from
Theorem~\ref{footprint-lowerbound}(vi).
\end{proof}

The next result gives an algebraic
formulation of the minimum distance of a projective Reed-Muller-type code 
in terms of the degree and the structure of the underlying
vanishing ideal.

\begin{theorem}\label{min-dis-vi} 
Let $K$ be a field and let $\mathbb{X}$ be a finite subset of
$\mathbb{P}^{s-1}$. If  $|\mathbb{X}|\geq 2$, then 
$$\delta_\mathbb{X}(d)=\delta_{I(\mathbb{X})}(d)\geq 1\ \mbox{ for
}d\geq 1.$$
\end{theorem}

\begin{proof} Setting $I=I(\mathbb{X})$, by Lemma~\ref{oct16-15}, 
the set $\mathcal{F}_d:=\{\, f\in S_d\colon
f\notin I,\, (I\colon f)\neq I\}$ is not empty for $d\geq 1$. 
Hence, using the formula for $V_\mathbb{X}(f)$ of 
Lemma~\ref{degree-formula-for-the-number-of-zeros-proj}, we
obtain
$$
\max\{|V_\mathbb{X}(f)|\colon {\rm ev}_d(f)\neq 0; f\in S_d\}=
\max\{\deg(S/(I,f))\vert\,
f\in \mathcal{F}_d\}.
$$
Therefore, using that $\deg(S/I)=|\mathbb{X}|$, we get    
\begin{eqnarray*}
\delta_{\mathbb{X}}(d)&=&\min\{\|{\rm ev}_d(f)\|
\colon {\rm ev}_d(f)\neq 0; f\in
S_d\}\\
&=&|\mathbb{X}|-\max\{|V_{\mathbb{X}}(f)|\colon {\rm ev}_d(f)\neq 0;
f\in S_d\}\\ 
&=&\deg(S/I)-\max\{\deg(S/(I,f))\vert\,
f\in\mathcal{F}_d\}=\delta_I(d),
\end{eqnarray*}
where $\|{\rm ev}_d(f)\|$ is the number of non-zero
entries of ${\rm ev}_d(f)$. 
\end{proof}

If $I$ is a graded ideal and $\Delta_\prec(I)\cap
S_d=\{t^{a_1},\ldots,t^{a_n}\}$, recall that 
$\mathcal{F}_{\prec, d}$ is the set of homogeneous 
standard polynomials of $S/I$ of degree $d$ which are
zero-divisors of $S/I$:
$$
\mathcal{F}_{\prec, d}:=\left.\left\{f=\textstyle\sum_i{\lambda_i}t^{a_i}\, \right|\,
f\neq 0,\ \lambda_i\in
K,\, (I\colon f)\neq I\right\}.
$$

The next result gives a description of the minimum distance which 
is suitable for computing this number using a computer algebra
system such as {\em Macaulay\/}$2$ \cite{mac2}. 

\begin{corollary}\label{oct15-15-1} If $K=\mathbb{F}_q$, $|\mathbb{X}|\geq 2$,
$I=I(\mathbb{X})$, and $\prec$ a monomial order, then
$$
\delta_\mathbb{X}(d)=\deg S/I-\max\{\deg(S/(I,f))\vert\,
f\in \mathcal{F}_{\prec, d}\}\geq 1\ \mbox{
for }\ d\geq 1.
$$
\end{corollary}

\begin{proof} It follows from Theorem~\ref{min-dis-vi} because by the
division algorithm any polynomial $f\in S_d$ can be written as $f=g+h$,
where $g$ is in $I_d$ and $h$ is a $K$-linear combination of standard
monomials of degree $d$. Notice that $(I\colon f)=(I\colon h)$. 
\end{proof}

The expression for $\delta_\mathbb{X}(d)$ of Corollary~\ref{oct15-15-1}
gives and algorithm that can be implemented in {\em Macaulay\/}$2$
\cite{mac2} to compute $\delta_\mathbb{X}(d)$ (see
Example~\ref{example-min-dis}). However, in practice, we can only find the minimum distance for 
small values of $q$ and $d$. Indeed, 
if $n=|\Delta_\prec(I)\cap
S_d|$, to compute $\delta_{I(\mathbb{X})}$ requires to test the inequality
$(I(\mathbb{X})\colon f)\neq I(\mathbb{X})$ and compute the
corresponding degree of $S/(I(\mathbb{X}),f)$ for the $n^q-1$
standard polynomials of $S/I$. 

\begin{corollary}\label{footprint-lowerbound-vanishing}
Let $K$ be a field, let $\prec$ be a monomial order, and let
$\mathbb{X}$ be a finite subset 
of $\mathbb{P}^{s-1}$. If 
$t_i$ is a zero-divisor of $S/I(\mathbb{X})$ for $i=1,\ldots,s$, 
then $\delta_\mathbb{X}(d)\geq {\rm fp}_{I(\mathbb{X})}(d)\geq 0$ 
for $d\geq 1$.
\end{corollary}

\begin{proof} The inequalities $\delta_\mathbb{X}(d)\geq {\rm
fp}_{I(\mathbb{X})}(d)\geq 0$ follow from 
Theorems~\ref{footprint-lowerbound} and \ref{min-dis-vi}. 
\end{proof}

One can use Corollary~\ref{footprint-lowerbound-vanishing} to estimate
the minimum distance of any Reed-Muller-type code over a set
$\mathbb{X}$ parameterized by a set of 
relatively prime monomials because in
this case $t_i$ is
a zero-divisor of $S/I(\mathbb{X})$ for $i=1,\ldots,s$ and one has the
following result that can be used to compute the vanishing ideal of
$\mathbb{X}$ using Gr\"obner bases and elimination theory. 

\begin{theorem}{\cite{vanishing-ideals}}\label{vanishing-ideals-theo}
Let $K=\mathbb{F}_q$ be a finite field. If $\mathbb{X}$ is a subset
of $\mathbb{P}^{s-1}$ parameterized by  
monomials $y^{v_1},\ldots,y^{v_s}$ in the variables $y_1,\ldots,y_n$, then 
$$ 
I(\mathbb{X})=(\{t_i-y^{v_i}z\}_{i=1}^s\cup\{y_i^{q}-y_i\}_{i=1}^n)\cap
S,
$$
and $I(\mathbb{X})$ is a binomial ideal.
\end{theorem}

As an application, Corollary~\ref{footprint-lowerbound-vanishing} will be 
used to study the minimum distance of projective nested
cartesian codes \cite{carvalho-lopez-lopez} over a 
set $\mathcal{X}$. In this case $t_i$ is
a zero-divisor of $S/I(\mathcal{X})$ for $i=1,\ldots,s$ and one has a
Gr\"obner basis for $I(\mathcal{X})$ \cite{carvalho-lopez-lopez} (see
Section~\ref{pncc-section}). 

\section{Degree formulas and some
inequalities}\label{degree-formulas-section}

Let $S=K[t_1,\ldots,t_s]$ be a polynomial ring over a field $K$, let
$d_1,\ldots,d_s$ be a non-decreasing sequence of positive integers
with $d_1\geq 2$ and $s\geq 2$, and let $L$ be the
ideal of $S$ generated by the set of all $t_it_j^{d_j}$ such that
$1\leq i<j\leq s$. In this section we show a formula for 
the degree of $S/(L,t^a)$ for any standard monomial $t^a$ of $S/L$.   

\begin{lemma}
\label{yuriko-vila} 
The ideal $L$ is Cohen-Macaulay of height $s-1$, has a unique 
irredundant primary decomposition
given by
$$
L=\mathfrak{q}_1\cap\cdots\cap\mathfrak{q}_s,
$$
where $\mathfrak{q}_i=
(t_1,\ldots,t_{i-1},t_{i+1}^{d_{i+1}},\ldots,t_s^{d_s})$ for $1\leq
i\leq s$, and $\deg(S/L)=1+\sum_{i=2}^sd_i\cdots d_s$.
\end{lemma}
\begin{proof} Using induction on $s$ and the depth lemma (see
\cite[Lemma~2.3.9]{monalg-rev}) it is seen that $L$ is 
Cohen-Macaulay. In particular $L$ is unmixed. Since the radical of $L$ is generated by all $t_it_j$
with $i<j$, the minimal primes of $L$ are
$\mathfrak{p}_1,\ldots,\mathfrak{p}_s$, where $\mathfrak{p}_i$ is
generated by $t_1,\ldots,t_{i-1},t_{i+1},\ldots,t_s$. The
$\mathfrak{p}_i$-primary component of $L$ is uniquely determined and
is given by $LS_{\mathfrak{p}_i}\cap S$. Inverting the
variable $t_i$ in $LS_{\mathfrak{p}_i}$ it follows that 
$LS_{\mathfrak{p}_i}=\mathfrak{q}_iS_{\mathfrak{p}_i}$. As
$\mathfrak{q}_i$ is an irreducible ideal, it is
$\mathfrak{p}_i$-primary and one has the equality $LS_{\mathfrak{p}_i}\cap
S=\mathfrak{q}_i$. By the additivity of the degree we obtain the
required formula for the degree of $S/L$.
\end{proof}

\begin{proposition}{\cite[Propositions~3.1.33 and 5.1.11]{monalg-rev}}\label{inter-tensoraa}
Let $A=R_1/I_1$, $B=R_2/I_2$ be two standard graded algebras over a field
$K$, where $R_1=K[\mathbf{x}]$, $R_2=K[\mathbf{y}]$ are polynomial 
rings in disjoint sets of variables and $I_i$ is an ideal 
of $R_i$. If $R=K[\mathbf{x},\mathbf{y}]$ and $I=I_1R+I_2R$, then 
$$
(R_1/I_1)\otimes_K(R_2/I_2)\simeq R/I\
\mbox{ and }\ F(A\otimes_K B,x)=F(A,x)F(B,x),
$$
where $F(A,x)$ and $F(B,x)$ are the Hilbert series of $A$ and $B$,
respectively.
\end{proposition}

\begin{proposition}\label{bounds-for-deg-init-f} Let 
$t^a=t_r^{a_r}\cdots t_s^{a_s}$ be a standard monomial of $S/L$ with
respect to a monomial order $\prec$. If $a_r\geq 1$, $a_i=0$ for
$i<r$, and $1\leq r\leq s$, 
then $0\leq a_i\leq d_i-1$ for $i>r$ and    
$$
\deg S/(L,t^a)=\left\{\begin{array}{ll}
\displaystyle\deg
S/L-\sum_{i=2}^{s}(d_i-a_i)\cdots(d_s-a_s)-1&\mbox{ if }
r=s,\, a_s\leq d_s,\\
\displaystyle\deg
S/L-1&\mbox{ if }
r=s,\, a_s\geq d_s+1,\\
\displaystyle\deg
S/L-\sum_{i=2}^{r+1}(d_i-a_i)\cdots(d_s-a_s)&\mbox{ if }
r<s,\, a_r\leq d_r,\\
\displaystyle\deg
S/L-(d_{r+1}-a_{r+1})\cdots(d_s-a_s)&\mbox{ if }
r<s,\, a_r\geq d_r+1.
\end{array}
\right.
$$
\end{proposition}
\begin{proof} As $f=t^a$ is not a multiple of $t_it_j^{d_j}$ for
$i<j$, we get that $0\leq a_i\leq d_i-1$ for $i>r$. 
To show the formula for the degree we proceed by induction 
on $s\geq 2$. In what follows we will use the theory of Hilbert
functions of graded algebras, as introduced in a classical paper of
Stanley \cite{Sta1}. In particular, we will freely use 
the additivity of Hilbert
series, a formula for
the Hilbert series of a complete intersection
\cite[Corollary~3.3]{Sta1}, the formula of
Lemma~\ref{yuriko-vila} for the degree of $S/L$, and the fact that any
monomial is a zero-divisor of $S/L$ (this follows from 
Lemma~\ref{yuriko-vila}).  
We split the proof of the case $s=2$ in
three easy cases. 

Case (1): Assume $s=2$, $r=1$. This case is independent of whether
$a_1\leq d_1$ or $a_1\geq d_1+1$ because the two possible values of
$\deg S/(L,f)$ coincide. There are exact sequences 
\begin{eqnarray*}
& & 0\longrightarrow
S/(t_1)[-d_2]
\stackrel{t_2^{d_2}}{\longrightarrow} S/(L,f)\longrightarrow
S/(t_2^{d_2},f)\longrightarrow 0,\\ 
& &\ \ \ \  0\longrightarrow S/(t_2^{a_2})[-a_1]
\stackrel{t_1^{a_1}}{\longrightarrow} S/(t_2^{d_2},f)\longrightarrow
S/(t_2^{d_2},t_1^{a_1})\longrightarrow 0.
\end{eqnarray*}
Taking Hilbert series we get
$$
F(S/(L,f),x)=\frac{x^{d_2}}{1-x}+\frac{x^{a_1}(1+x+\cdots+x^{a_2-1})}{1-x}+
\left(\sum_{i=0}^{d_2-1}x^i\right)\left(
\sum_{i=0}^{a_1-1}x^i\right).
$$
Writing $F(S/(L,f),x)=h(x)/(1-x)$ with $h(x)\in\mathbb{Z}[x]$ and
$h(1)>0$, and noticing that $h(1)$ is the degree of $S/(L,f)$, we
get 
$$
\deg S/(L,f)=1+a_2=(d_2+1)-(d_2-a_2)=\deg(S/L)-(d_2-a_2).
$$

Case (2): Assume $s=2$, $r=2$, $a_2\leq d_2$. In this case $(L,f)$ is equal to
$(t_2^{a_2})$. Thus
$$ 
\deg S/(L,f)=a_2=(1+d_2)-(d_2-a_2)-1=\deg S/L-(d_2-a_2)-1.
$$

Case (3): Assume $s=2$, $r=2$, $a_2\geq d_2+1$. Taking Hilbert series in
the exact sequence
$$
0\longrightarrow
S/(t_1,t_2^{a_2-d_2})[-d_2]
\stackrel{t_2^{d_2}}{\longrightarrow} S/(L,f)\longrightarrow
S/(t_2^{d_2})\longrightarrow 0, 
$$
we obtain
$$
F(S/(L,f),x)=x^{d_2}(1+x+\cdots+x^{a_2-d_2-1})+\frac{(1+x+\cdots+x^{d_2-1})}{1-x}.
$$
Thus we may proceed as in Case (1) to get
$\deg(S/(L,f))=d_2=\deg(S/L)-1$.

This completes the initial induction step. We may now assume that
$s\geq 3$ and split the proof in three cases. 

Case (I): Assume $r=s\geq 3$ and $a_s\leq d_s$. Thus $f=t_s^{a_s}$ and $a_i=0$ for $i<s$. 
Setting $L'$ equal to the ideal generated by
the set of all $t_it_j^{d_j}$ such that $2\leq i<j\leq s$, there is an 
exact sequence 
$$
0\longrightarrow
S/(t_2^{d_2},\ldots,t_{s-1}^{d_{s-1}},t_s^{a_s})[-1]
\stackrel{t_1}{\longrightarrow} S/(L,t_s^{a_s})\longrightarrow
S/(L',t_s^{a_s},t_1)\longrightarrow 0.
$$
Taking Hilbert series one has 
$$
F(S/(L,t_s^{a_s}),x)=xF(S/(t_2^{d_2},\ldots,t_{s-1}^{d_{s-1}},t_s^{a_s}),x)+F(S/(L',t_s^{a_s},t_1),x).
$$
Hence, setting $S'=K[t_2,\ldots,t_s]$, from the induction hypothesis applied to 
$S'/(L',t_s^{a_s})$, and using that $\deg S'/L'=
\deg S/L-d_2\cdots d_{s-1}d_s$ (see Lemma~\ref{yuriko-vila}), we obtain
\begin{eqnarray*}
\deg(S/(L,f))&=&d_2\cdots
d_{s-1}a_s+\deg(S'/L')-\sum_{i=3}^{s}d_i\cdots d_{s-1}(d_s-a_s)-1\\
&=&\deg(S/L)-\sum_{i=2}^{s}d_i\cdots d_{s-1}(d_s-a_s)-1.
\end{eqnarray*}

Case (II): Assume $r=s\geq 3$ and $a_s\geq d_s+1$. Using the exact sequence 
$$
0\longrightarrow
S/(t_2^{d_2},\ldots,t_{s-1}^{d_{s-1}},t_s^{d_s})[-1]
\stackrel{t_1}{\longrightarrow} S/(L,t_s^{a_s})\longrightarrow
S/(L',t_s^{a_s},t_1)\longrightarrow 0,
$$
we can proceed as in Case (I) to get $\deg(S/(L,f))=\deg(S/L)-1$.

Case (III): Assume $r<s$. Then, by assumption, $a_s<d_s$. Let
$L'$ be the
ideal generated by the set of all $t_it_j^{d_j}$ such that $1\leq
i<j\leq s-1$. Setting $f'=t_r^{a_r}\cdots
t_{s-1}^{a_{s-1}}$ and $S'=K[t_1,\ldots,t_{s-1}]$, there are exact sequences 
\begin{eqnarray*}
& & 0\longrightarrow
S/(t_1,\ldots,t_{s-1})[-d_s]
\stackrel{t_s^{d_s}}{\longrightarrow} S/(L,f)\longrightarrow
S/(L',f,t_s^{d_s})\longrightarrow 0,\\ 
& &\ \ \ \  0\longrightarrow S/(L',f',t_s^{d_s-a_s})[-a_s]
\stackrel{t_s^{a_s}}{\longrightarrow} S/(L',f,t_s^{d_s})\longrightarrow
S/(L',t_s^{a_s})\longrightarrow 0.
\end{eqnarray*}
Hence taking Hilbert series, and applying
Proposition~\ref{inter-tensoraa}, we get
$$ 
F(S/(L,f),x)=
\frac{x^{d_s}}{1-x}+F(S'/(L',f'),x)F(K[t_s]/(t_s^{d_s-a_s}),x)+F(S'/L',x)F(K[t_s]/(t_s^{a_s}),x).
$$
Writing $F(S/(L,f),x)=h(x)/(1-x)$ with $h(x)\in\mathbb{Z}[x]$ and
$h(1)>0$, and noticing that $h(1)$ is the degree of $S/(L,f)$, the
induction hypothesis applied to $S'/(L',f')$ yields the equality 
\begin{eqnarray*}
\deg S/(L,f)&=&1+\left(\deg 
S'/L'-\sum_{i=2}^{r+1}(d_i-a_i)\cdots(d_{s-1}-a_{s-1})\right)(d_s-a_s)+(\deg
S'/L')a_s\\
&=&1+(\deg S'/L')d_s-\sum_{i=2}^{r+1}(d_i-a_i)\cdots(d_{s}-a_{s})\ 
\mbox{ if }\ a_r\leq d_r.
\end{eqnarray*}
or the equality 
\begin{eqnarray*}
\deg S/(L,f)&=&1+\left(\deg 
S'/L'-(d_{r+1}-a_{r+1})\cdots(d_{s-1}-a_{s-1})\right)(d_s-a_s)+
\deg(S'/L')a_s\\
&=&1+\deg(
S'/L')d_s-(d_{r+1}-a_{r+1})\cdots(d_{s}-a_{s})\ 
\mbox{ if }\ a_r\geq d_r+1.
\end{eqnarray*}

To complete the proof it suffices to notice that $\deg(S/L)=1+\deg
(S'/L')d_s$. This equality follows readily from
Lemma~\ref{yuriko-vila}.
\end{proof}

\begin{remark} Cases (1), (2), and (3) can
also be shown using Hilbert functions instead of Hilbert series, but 
case (III) is easier to handle using Hilbert series. 
\end{remark}

\begin{lemma}\label{basic-ineq}
Let $a_1,\ldots,a_r,a,b,e$ be positive integers with $e\geq a$. Then
\begin{itemize}
\item[(a)] $a_1\cdots a_r\geq(a_1+\cdots+a_r)-(r-1)$, and
\item[(b)] $a(e-b)\geq (a-b)e$.
\end{itemize}
\end{lemma}

\begin{proof} Part (a) follows by
induction on $r$, and part
(b) is straightforward. 
\end{proof}

The next inequality is a generalization of part (a).

\begin{lemma}\label{pepe-vila} Let $1\leq e_1\leq\cdots\leq e_m$ and $0\leq b_i\leq e_i-1$
for $i=1,\ldots,m$ be integers. Then 
\begin{equation}\label{aug-18-15}
\prod_{i=1}^m(e_i-b_i)\geq\left(\sum_{i=1}^k(e_i-b_i)-(k-1)-\sum_{i=k+1}^mb_i\right)e_{k+1}\cdots e_m
\end{equation}
for $k=1,\ldots,m$, where $e_{k+1}\cdots e_m=1$ and
$\sum_{i=k+1}^mb_i=0$ if $k=m$.
\end{lemma}

\demo 
Fix $m$ and $1\leq k\leq m$. We will proceed by
induction on $\sigma=\sum_{i=1}^k(e_i-b_i-1)$. If $\sigma=0$, then
$e_i-b_i-1=0$ for $i=1,\ldots,k$. Thus either $1-\sum_{i=k+1}^m{b_i}<0$
or $1-\sum_{i=k+1}^m{b_i}\geq 1$. In the first case the inequality is
clear because the left hand side of Eq.~(\ref{aug-18-15}) is positive and
in the second case one has $b_i=0$ for $i=k+1,\ldots,m$ and equality
holds in Eq.~(\ref{aug-18-15}). Assume that $\sigma>0$.  If $k=m$ or $b_i=0$ for
$i=k+1,\ldots,m$, the inequality follows at once from
Lemma~\ref{basic-ineq}(a). 
Thus, we may assume $k<m$ and
$b_j>0$ for some $k+1\leq j\leq m$. To simplify notation, and 
without loss of generality, we may assume that $j=m$, that is, 
$b_m>0$. 
 If the right hand side of Eq.~(\ref{aug-18-15}) is
negative or zero, the inequality holds. Thus we may also assume that
\begin{equation}\label{aug-18-15-1}
\sum_{i=1}^k(e_i-b_i)-\sum_{i=k+1}^mb_i\geq k.
\end{equation}
Hence there is $1\leq \ell\leq k$ such that $e_\ell-b_\ell\geq 2$.

Case (1): Assume $e_\ell-b_\ell-b_m\geq 1$. Setting
$a=e_\ell-b_\ell$, $e=e_m$, and $b=b_m$ in Lemma~\ref{basic-ineq}(b), we
get 
\begin{equation}\label{aug-18-15-2}
(e_\ell-b_\ell)(e_m-b_m)\geq (e_\ell-(b_\ell+b_m))e_m.
\end{equation}
Therefore using Eq.~(\ref{aug-18-15-2}), and then applying the
induction hypothesis to the two sequences of integers
$$
e_1,\ldots,e_{\ell-1},e_\ell,e_{\ell+1},\ldots,e_{m-1},e_m;\ \ 
b_1,\ldots,b_{\ell-1},b_\ell+b_m,b_{\ell+1},\ldots,b_{m-1},0,
$$
we get the inequalities
\begin{eqnarray*}
\prod_{i=1}^m(e_i-b_i)&=&\left(\prod_{i\notin\{\ell,m\}}(e_i-b_i)\right)(e_\ell-b_\ell)(e_m-b_m)\\
&\geq&\left(\prod_{i\notin\{\ell,m\}}(e_i-b_i)\right)
(e_\ell-(b_\ell+b_m))e_m
\\ &\geq&\left(\sum_{\ell\neq
i=1}^k(e_i-b_i)+(e_\ell-(b_\ell+b_m))-(k-1)-\sum_{i=k+1}^{m-1}b_i\right)e_{k+1}\cdots
e_m\\  
&=&\left(\sum_{i=1}^k(e_i-b_i)-(k-1)-\sum_{i=k+1}^mb_i\right)e_{k+1}\cdots
e_m.
\end{eqnarray*}

Case (2): Assume $e_\ell-b_\ell-b_m<1$. Setting
$r_\ell=e_\ell-b_\ell-1\geq 1$, one has 
$$
b_\ell+r_\ell=e_\ell-1\geq 1,\,\,\ b_m-r_\ell\geq 1,\,\,\ 
e_\ell-(b_\ell+r_\ell)=1.
$$
On the other hand, by Lemma~\ref{basic-ineq}(a), one has 
\begin{equation}\label{aug-18-15-3}
(e_\ell-b_\ell)(e_m-b_m)\geq
(e_\ell-b_\ell)+(e_m-b_m)-1=(e_\ell-(b_\ell+r_\ell))(e_m-(b_m-r_\ell)).
\end{equation}
Therefore using Eq.~(\ref{aug-18-15-3}), and then applying the
induction hypothesis to the two sequences of integers
$$
e_1,\ldots,e_{\ell-1},e_\ell,e_{\ell+1},\ldots,e_{m-1},e_m;\ \ 
b_1,\ldots,b_{\ell-1},b_\ell+r_\ell,b_{\ell+1},\ldots,b_{m-1},b_m-r_\ell,
$$
we get the inequalities
\begin{eqnarray*}
\prod_{i=1}^m(e_i-b_i)&=&\left(\prod_{i\notin\{\ell,m\}}(e_i-b_i)\right)(e_\ell-b_\ell)(e_m-b_m)
\\
&\geq&\left(\prod_{i\neq\{\ell,m\}}(e_i-b_i)\right)
(e_\ell-(b_\ell+r_\ell))(e_m-(b_m-r_\ell))
\\ &\geq&\left(\sum_{\ell\neq
i=1}^k(e_i-b_i)+(e_\ell-(b_\ell+r_\ell))-(k-1)-\sum_{i=k+1}^{m-1}b_i-(b_m-r_\ell)\right)e_{k+1}\cdots
e_m\\  
&=&\left(\sum_{i=1}^k(e_i-b_i)-(k-1)-\sum_{i=k+1}^mb_i\right)e_{k+1}
\cdots e_m.\ \Box
\end{eqnarray*}

\begin{proposition}\label{aug-28-15}
Let $1\leq e_1\leq\cdots\leq e_m$ and $0\leq b_i\leq e_i-1$
for $i=1,\ldots,m$ be integers. If $b_0\geq 1$, then 
\begin{equation}\label{aug-27-15-2}
\prod_{i=1}^m(e_i-b_i)\geq
\left(\sum_{i=1}^{k+1}(e_i-b_i)-(k-1)-b_0-\sum_{i=k+2}^m b_i\right)e_{k+2}\cdots e_m
\end{equation}
for $k=0,\ldots,m-1$, where $e_{k+2}\cdots e_m=1$ and
$\sum_{i=k+2}^mb_i=0$ if $k=m-1$.
\end{proposition}

\begin{proof} If $0\leq k\leq m-1$, then $1\leq k+1\leq m$. Applying
Lemma~\ref{pepe-vila}, and making the substitution $k\rightarrow k+1$ in
Eq.~(\ref{aug-27-15-2}), we get
\begin{eqnarray*}
\prod_{i=1}^m(e_i-b_i)&\geq&\left(\sum_{i=1}^{k+1}(e_i-b_i)-
k-\sum_{i=k+2}^mb_i\right)e_{k+2}\cdots e_m
\\
&\geq&\left(\sum_{i=1}^{k+1}(e_i-b_i)-
(k-1)-b_0-\sum_{i=k+2}^mb_i\right)e_{k+2}\cdots e_m,
\end{eqnarray*}
where the second inequality holds because $b_0\geq 1$.
\end{proof}

\section{Projective nested cartesian codes}\label{pncc-section}

In this section we introduce projective nested cartesian codes, a type of evaluation codes that 
generalize the classical projective Reed--Muller
codes \cite{sorensen}. As an application we will give some support to
a conjecture of Carvalho, Lopez-Neumann and L\'opez.

Let $K=\mathbb{F}_q$ be a finite field, let $A_1,\ldots,A_s$ be a collection of subsets 
of $K$, and let 
$$
\mathcal{X}=[A_1\times\cdots\times A_s]
$$
be the image of $A_1\times\cdots\times A_s\setminus\{0\}$ under the map 
$K^s\setminus\{0\}\rightarrow\mathbb{P}^{s-1}$, $x\rightarrow [x]$. 

\begin{definition}{\cite{carvalho-lopez-lopez}}\label{pncc}
The set $\mathcal{X}$ is called a
{\it projective nested cartesian set} if  
\begin{itemize}
\item[(i)] $\{0,1\}\subset A_i$ for $i=1,\ldots,s$, 
\item[(ii)] $a/b\in A_j$ for  $1\leq i<j\leq s$, 
$a\in A_j$, $0\neq b\in A_i$, and 
\item[(iii)] $d_1\leq\cdots\leq d_s$, where $d_i=|A_i|$ for
$i=1,\ldots,s$.  
\end{itemize}
If $\mathcal{X}$ is a projective nested cartesian set, we call 
$C_{\mathcal{X}}(d)$ a {\it projective nested cartesian
code}. 
\end{definition}

\begin{conjecture}{\rm(Carvalho, Lopez-Neumann, and L\'opez 
\cite{carvalho-lopez-lopez})}\label{carvalho-lopez-lopez-conjecture}
Let $C_\mathcal{X}(d)$ be the $d$-th projective nested
cartesian code on the set
$\mathcal{X}=[A_1\times\cdots\times A_s]$ with $d_i=|A_i|$ for
$i=1,\ldots,s$. Then its minimum distance is given by 
$$
\delta_\mathcal{X}(d)=\left\{\hspace{-1mm}
\begin{array}{ll}\left(d_{k+2}-\ell+1\right)d_{k+3}\cdots d_s&\mbox{ if }
d\leq \sum\limits_{i=2}^{s}\left(d_i-1\right),\\
\qquad \qquad 1&\mbox{ if } d\geq
\sum\limits_{i=2}^{s}\left(d_i-1\right)+1,
\end{array}
\right.
$$
where $0\leq k\leq s-2$ and $\ell$ are the unique integers such that 
$d=\sum_{i=2}^{k+1}\left(d_i-1\right)+\ell$ and $1\leq \ell \leq
d_{k+2}-1$. 
\end{conjecture}

In what follows $\mathcal{X}=[A_1\times\cdots\times A_s]$ denotes a
projective nested cartesian set and $C_\mathcal{X}(d)$ is its 
corresponding $d$-th projective Reed-Muller-type code. Throughout this
section $\prec$ is the lexicographical order on $S$ with 
$t_1\prec\cdots \prec t_s$ and ${\rm in}_\prec(I(\mathcal{X}))$ is the
initial ideal of $I(\mathcal{X})$.

\begin{proposition}{\cite{carvalho-lopez-lopez}}\label{carvalho-deg-reg}
The initial ideal ${\rm in}_\prec(I(\mathcal{X}))$ is 
generated by the set of all monomials $t_it_j^{d_j}$ such that
$1\leq i<j\leq s$, 
$$\deg(S/I(\mathcal{X}))=1+\sum_{i=2}^sd_i\cdots
d_s,\ \mbox{ and } \ {\rm reg}(S/I(\mathcal{X}))=
1+\sum_{i=2}^s(d_i-1).
$$
\end{proposition}

Carvalho, Lopez-Neumann and L\'opez, showed the conjecture when the
$A_i$'s are subfields of $\mathbb{F}_q$. They also showed that the conjecture can be
reduced to:

\begin{conjecture}{\rm (Carvalho, Lopez-Neumann, and L\'opez
\cite{carvalho-lopez-lopez})}
\label{carvalho-lopez-lopez-conjecture-new} 
If $0\neq f\in S_d$ is a standard polynomial, with respect to
$\prec$, such that $(I(\mathcal{X}):f)\neq I(\mathcal{X})$ and 
$1\leq d\leq\sum_{i=2}^{s}(d_i-1)$, then  
$$
|V_\mathcal{X}(f)|\leq \deg(
S/I(\mathcal{X}))-\left(d_{k+2}-\ell+1\right)d_{k+3}\cdots d_s,
$$
where $0\leq k\leq s-2$ and $\ell$ are integers such that 
$d=\sum_{i=2}^{k+1}\left(d_i-1\right)+\ell$ and $1\leq \ell \leq
d_{k+2}-1$. 
\end{conjecture}

We show a degree formula and use this to give an upper
bound for $|V_\mathcal{X}(f)|$.

\begin{theorem}\label{bounds-for-deg-init-f-1} Let $\prec$ be the lexicographical order on $S$ with 
$t_1\prec\cdots \prec t_s$ and let $f\neq 0$ be a standard polynomial
with ${\rm in}_\prec(f)=t_r^{a_r}\cdots t_s^{a_s}$ and $a_r\geq 1$.
Then $0\leq a_i\leq d_i-1$ for $i>r$ and 
\begin{eqnarray*}
|V_{\mathcal{X}}(f)|&\leq& \deg(S/({\rm
in}_\prec(I(\mathcal{X})),{\rm in}_\prec(f)))\\
&=&
\left\{\begin{array}{ll}\displaystyle\deg(
S/I(\mathcal{X}))-\sum_{i=2}^{r+1}(d_i-a_i)\cdots(d_s-a_s)&\mbox{ if } 
a_r\leq d_r,\\
\displaystyle\deg(
S/I(\mathcal{X}))-(d_{r+1}-a_{r+1})\cdots(d_s-a_s)&\mbox{ if }
a_r\geq d_r+1,
\end{array}
\right.
\end{eqnarray*}
where $(d_i-a_i)\cdots(d_s-a_s)=1$ if $i>s$ and $a_i=0$ for $i<r$.
\end{theorem}

\begin{proof} By Proposition~\ref{carvalho-deg-reg} 
the initial ideal of $I(\mathcal{X})$ is
generated by the set of all $t_it_j^{d_j}$ such that
$1\leq i<j\leq s$ and the degree of $S/({\rm in}_\prec(I(\mathcal{X}))$
is equal to the degree 
of $S/I(\mathcal{X})$. As ${\rm in}_\prec(f)$ is a standard monomial,
it follows that $0\leq a_i\leq d_i-1$ for $i>r$. Notice that if $f$
is not a zero-divisor of $S/I(\mathcal{X})$, then $V_{\mathcal{X}}(f)=\emptyset$.
Thus the inequality follows at once from
Corollary~\ref{poly-bounds-initial} and the equality follows 
from Proposition~\ref{bounds-for-deg-init-f}. 
\end{proof}

The next result gives some support for
Conjecture~\ref{carvalho-lopez-lopez-conjecture-new}. 

\begin{theorem}\label{yuriko-pepe-vila}
Let $\prec$ be the lexicographical order on $S$ with 
$t_1\prec\cdots \prec t_s$. If $0\neq f\in S_d$ is a standard
polynomial such that $1\leq d\leq\sum_{i=2}^{s}(d_i-1)$ and $t_1$
divides  
${\rm in}_\prec(f)$, then  
$$
|V_\mathcal{X}(f)|\leq \deg(
S/I(\mathcal{X}))-\left(d_{k+2}-\ell+1\right)d_{k+3}\cdots d_s,
$$
where $0\leq k\leq s-2$ and $\ell$ are integers such that 
$d=\sum_{i=2}^{k+1}\left(d_i-1\right)+\ell$ and $1\leq \ell \leq
d_{k+2}-1$. 
\end{theorem}

\begin{proof} By
Lemma~\ref{degree-formula-for-the-number-of-zeros-proj} we may assume
that $(I(\mathcal{X}):f)\neq I(\mathcal{X})$. Let  $t^a={\rm
in}_\prec(f)$ be the initial monomial of 
$f$. By Proposition~\ref{carvalho-deg-reg}, we can
write 
$$
t^a=t_1^{a_1}\cdots t_s^{a_s},
$$
with $a_1\geq 1$, $0\leq a_i\leq d_i-1$ for
$i>1$. 
By Lemmas~\ref{degree-formula-for-the-number-of-zeros-proj} and 
\ref{degree-initial-footprint} it suffices to show that the
following inequality holds
\begin{equation}\label{dec14-15}
\deg(S/({\rm in}_\prec(I(\mathcal{X})),t^a))\leq 
\deg(S/I(\mathcal{X}))
-\left(d_{k+2}-\ell+1\right)d_{k+3}\cdots 
d_s.
\end{equation}

If we substitute $-\ell=\sum_{i=2}^{k+1}(d_i-1)-\sum_{i=1}^sa_i$ in
Eq.~(\ref{dec14-15}), and use the formula for the degree of 
$S/({\rm in}_\prec(I(\mathcal{X})),t^a)$ given in
Theorem~\ref{bounds-for-deg-init-f-1}, we need only show that the
following inequalities hold for $r=1$: 
\begin{eqnarray}
& &\sum_{i=2}^{r+1}(d_i-a_i)\cdots(d_s-a_s)\geq\label{dec14-1} \\ 
& &\ \ \ \ \ \ \ \ \ \ \ \ \ \ \ \ \ \ \left(\sum_{i=2}^{k+2}(d_i-a_i)-(k-1)-a_1-\sum_{i=k+3}^sa_i\right)d_{k+3}\cdots
d_s\mbox{ if }a_r\leq d_r,\nonumber\\ 
& &\prod_{i=r+1}^{s}(d_i-a_i) \geq\left(\sum_{i=2}^{k+2}(d_i-a_i)-(k-1)-a_1-\sum_{i=k+3}^sa_i\right)d_{k+3}\cdots
d_s\mbox{ if }a_r\geq d_r+1,\label{dec14-2}
\end{eqnarray}
for $0\leq k\leq s-2$, where $(d_i-a_i)\cdots(d_s-a_s)=1$ if $i>s$ and $a_i=0$ for $i<r$. 

Assume $r=1$. Then Eqs.~(\ref{dec14-1}) and (\ref{dec14-2})
are the same. Thus we need only show the inequality 

\begin{equation*}
\prod_{i=2}^{s}(d_i-a_i) \geq\left(\sum_{i=2}^{k+2}(d_i-a_i)-(k-1)-a_1-\sum_{i=k+3}^sa_i\right)d_{k+3}\cdots
d_s,
\end{equation*}
for $0\leq k\leq s-2$. This inequality follows making $m=s-1$, $e_i=d_{i+1}$, $b_i=a_{i+1}$
for $i=1,\ldots,m$, and $b_0=a_1$ in Proposition~\ref{aug-28-15}.
\end{proof}

The following resembles
Conjecture~\ref{carvalho-lopez-lopez-conjecture} when $d=1$:

\begin{conjecture}\cite[Conjecture~4.9]{tohaneanu-vantuyl} Let $\mathbb{X}$ be a
finite set of points in $\mathbb{P}^{s-1}$. If $I(\mathbb{X})$ is a
complete intersection generated by $f_1,\ldots,f_{s-1}$, with 
$e_i=\deg(f_i)$ for $i=1,\ldots,s-1$, and $2\leq e_i\leq e_{i+1}$ for
all $i$, then $\delta_\mathbb{X}(1)\geq (e_1-1)e_2\cdots e_{s-1}$.
\end{conjecture}

Next we show that the corresponding conjecture is true for projective nested
cartesian codes by proving that
Conjecture~\ref{carvalho-lopez-lopez-conjecture} is true for $d=1$. 
\begin{proposition} $\delta_\mathcal{X}(1)=d_2\cdots d_s$. 
\end{proposition}
\begin{proof} It suffices to show that
Conjecture~\ref{carvalho-lopez-lopez-conjecture-new} is true for
$d=1$, that is, if $0\neq f\in S_1$ is a standard polynomial 
such that $(I(\mathcal{X}):f)\neq I(\mathcal{X})$, then we
must show the inequality  
$$
|V_\mathcal{X}(f)|\leq \deg(S/I(\mathcal{X}))-d_2\cdots d_s.
$$

As ${\rm in}_\prec(f)=t_r$ for some $1\leq r\leq s$, we can 
write ${\rm in}_\prec(f)=t_1^{a_1}\cdots t_r^{a_r}\cdots t_s^{a_s}$,
where $a_r=1$, $a_i=0$ for 
$i\neq r$, and $2\leq d_2\leq\cdots\leq d_s$. Therefore, by 
Theorem~\ref{bounds-for-deg-init-f-1}, we get
$$
|V_{\mathcal{X}}(f)|\leq\deg(
S/I(\mathcal{X}))-\sum_{i=2}^{r+1}(d_i-a_i)\cdots(d_s-a_s),
$$
where $(d_i-a_i)\cdots(d_s-a_s)=1$ if $i>s$. Hence the proof reduces to
showing the inequality
$$
\sum_{i=2}^{r+1}(d_i-a_i)\cdots(d_s-a_s)\geq d_2\cdots d_s.
$$

If $r=s$, this inequality
follows readily by induction on $s\geq 2$. If $r<s$ the inequality
also follows by induction on $s$ by noticing that, in this case, both
sides of the inequality have $d_s$ as a common factor because $d_s$
appears in all terms of the summation of the left hand side.
\end{proof}

Let $\mathcal{L}_d$ be the $K$-vector space 
generated by all $t^a\in S_d$ such that $t_1$ divides $t^a$ 
and let $C_d$ be the image of $\mathcal{L}_d$ under the
evaluation map ${\rm ev}_d$. From the next result it follows that the
minimum distance of $C_\mathcal{X}(d)$ 
proposed in Conjecture~\ref{carvalho-lopez-lopez-conjecture} is in
fact the minimum distance of the evaluation linear code
$C_d$.

\begin{corollary}\label{yuriko-pepe-vila-coro}
Let $\mathcal{L}_d$ be the $K$-vector space 
generated by all $t^a\in S_d$ such that $t^a$ contains $t_1$. 
If $1\leq d\leq\sum_{i=2}^{s}(d_i-1)$, then  
$$
\max\{|V_\mathcal{X}(f)|:\, f\notin I(\mathcal{X}),\,
f\in\mathcal{L}_d\} =\deg(S/I(\mathcal{X}))-\left(d_{k+2}-\ell+1\right)d_{k+3}\cdots d_s,
$$
where $0\leq k\leq s-2$ and $\ell$ are integers, 
$d=\sum_{i=2}^{k+1}\left(d_i-1\right)+\ell$, and $1\leq \ell \leq
d_{k+2}-1$. 
\end{corollary} 

\begin{proof} Take $f\in\mathcal{L}_d\setminus I(\mathcal{X})$. 
Let $\prec$ be the lexicographical order with $t_1\prec\cdots\prec t_s$ and
let $\mathcal{G}$ be the Gr\"obner basis of $I(\mathcal{X})$ given in
\cite[Proposition~2.14]{carvalho-lopez-lopez}. By the
division algorithm, we can write $f=\sum_{i=1}^ra_ig_i+g$, where
$g_i\in\mathcal{G}$ for all $i$ and $g$ is a standard polynomial of
degree $d$. The polynomial $g$ is again in
$\mathcal{L}_d\setminus I(\mathcal{X})$. Indeed if
$g\notin\mathcal{L}_d$, there is at least one monomial
of $g$ that do not contain $t_1$, then 
making $t_1=0$ in the last equality, we get an
equality of the form $0=\sum_{i=1}^rb_ig_i+h$, where $h$ is a non-zero
standard polynomial of $I(\mathcal{X})$, a contradiction. Hence, by
Theorem~\ref{yuriko-pepe-vila}, the inequality $\leq$ follows because
$|V_\mathcal{X}(f)|=|V_\mathcal{X}(g)|$. To show 
equality notice that according to the proof of
\cite[Lemma~3.1]{carvalho-lopez-lopez}, there is a polynomial $f$ of
degree $d$ in 
$\mathcal{L}_d\setminus I(\mathcal{X})$ whose number of zeros in $\mathcal{X}$ is equal to the right
hand side of the required equality. 
\end{proof}

\section{Examples}\label{examples-section}

In this section we show some examples that illustrates how some of our results
can be used in practice. 

\begin{example}\label{dec4-15}
Let $K$  be the field $\mathbb{F}_3$, let
$\mathbb{X}$ be the subset of $\mathbb{P}^3$ given by 
$$
\mathbb{X}=\{[e_1],\,
[e_2],\,[e_3],\,[e_4],\,[(1,-1,-1,1)],\,[(1,1,1,1)],\,[(-1,-1,1,1)],\,[(-1,1,-1,1)]\},
$$
where $e_i$ is the $i$-th unit vector, and let $I=I(\mathbb{X})$ be
the vanishing 
ideal of $\mathbb{X}$. Using Lemma~\ref{primdec-ix-a} and 
{\em Macaulay\/}$2$ \cite{mac2}, 
we get that $I$ is the ideal of $S=K[t_1,t_2,t_3,t_4]$ generated by
the binomials $t_1t_2-t_3t_4,\, t_1t_3-t_2t_4,\,t_2t_3-t_1t_4$. 
Hence, using Theorem~\ref{wolmer-obs} and the procedure below
for {\em Macaulay\/}$2$ \cite{mac2}, we get 
\begin{eqnarray*}
\hspace{-11mm}&&\left.
\begin{array}{c|c|c|c|c}
d & 1 & 2 & 3 & \cdots\\
   \hline
{\deg}(S/I) & 8 & 8 & 8 &\cdots
 \\ 
   \hline
 H_I(d)    \    & 4 & 7  & 8& 
 \cdots\\ 
   \hline
 \delta_I(d) &4& 2& 1& \cdots
\end{array}
\right.
\end{eqnarray*}

\begin{verbatim}
q=3
S=ZZ/q[t1,t2,t3,t4]
I=ideal(t1*t2-t3*t4,t1*t3-t2*t4,t2*t3-t1*t4)
M=coker gens gb I
h=(d)->min apply(apply(apply(apply(toList
(set(0..q-1))^**(hilbertFunction(d,M))-
(set{0})^**(hilbertFunction(d,M)),toList),x->basis(d,M)*vector x),
z->ideal(flatten entries z)),x-> degree quotient(I,x))
apply(1..2,h)--this gives the minimum distance in degrees 1,2  
\end{verbatim}
\end{example}

\begin{example}\label{example-min-dis} Let $\mathbb{X}$ be the set in $\mathbb{P}^3$
parameterized by $y_1y_2,y_2y_3,y_3y_4,y_1y_4$ over the field
$\mathbb{F}_3$. Using Corollary~\ref{oct15-15-1}, Theorem~\ref{vanishing-ideals-theo}, and the
following procedure for {\em Macaulay\/}$2$
\cite{mac2} we get 

\begin{eqnarray*}
\hspace{-11mm}&&\left.
\begin{array}{c|c|c|c|c}
d & 1 & 2 & 3 & \cdots\\
   \hline
 |\mathbb{X}| & 16 & 16 & 16 &\cdots
 \\ 
   \hline
 H_\mathbb{X}(d)    \    & 4 & 9  & 16& 
 \cdots\\ 
   \hline
 \delta_{\mathbb{X}}(d) &9& 4& 1& \cdots \\ 
\hline
 {\rm fp}_{I(\mathbb{X})}(d) &6& 3& 1& \cdots \\ 
\end{array}
\right.
\end{eqnarray*}

\begin{verbatim}
q=3
R=ZZ/q[y1,y2,y3,y4,z,t1,t2,t3,t4,MonomialOrder=>Eliminate 5];
f1=y1*y2, f2=y2*y3, f3=y3*y4, f4=y4*y1
J=ideal(y1^q-y1,y2^q-y2,y3^q-y3,y4^q-y4,t1-f1*z,t2-f2*z,t3-f3*z,t4-f4*z)
C4=ideal selectInSubring(1,gens gb J)
S=ZZ/q[t1,t2,t3,t4];
I=sub(C4,S)
M=coker gens gb I 
h=(d)->degree M - max apply(apply(apply(apply(
toList (set(0..q-1))^**(hilbertFunction(d,M))-
(set{0})^**(hilbertFunction(d,M)), toList),x->basis(d,M)*vector x),
z->ideal(flatten entries z)),x-> if not
quotient(I,x)==I then degree ideal(I,x) else 0)--The function h(d) 
--gives the minimum distance in degree d 
h(1), h(2)
f=(x1) -> degree ideal(x1,leadTerm gens gb I)
fp=(d)->degree M - max apply(flatten entries basis(d,M),f)--The
--function fp(d) gives the footprint in degree d
L=toList(1..regularity M) 
apply(L,fp)
\end{verbatim}
\end{example}

\begin{example}\label{footprint-bad}
Let $\mathbb{X}$ be a projective torus in $\mathbb{P}^2$ over the
field $K=\mathbb{F}_3$. The vanishing ideal $I=I(\mathbb{X})$ is
generated by $t_1^2-t_3^2$, $t_2^2-t_3^2$. The polynomial $F=(t_1-t_2)^d$
is a zero-divisor of $S/I$ because $(t_1-t_3,t_2-t_3)$ is an associated
prime of $S/I$ and $F\notin I$ because $F$ does not vanish at
$[(1,-1)]$. Hence, $\mathcal{F}_d\neq\emptyset$. If
$\prec$ is the lexicographical order $t_1\succ t_2\succ t_3$, then
$t_3^d$ is a standard monomial which is not a zero-divisor of $S/I$ and
$S/{\rm in}_\prec(I)$. This causes 
$\deg(S/({\rm in}_\prec(I),t_3^2))$ to be greater than $\deg(S/I)$.
Using {\em Macaulay\/}$2$ \cite{mac2} we obtain 
\begin{eqnarray*}
\hspace{-11mm}&&\left.
\begin{array}{c|c|r|c}
d & 1 & 2 & \cdots\\
   \hline
 |\mathbb{X}| & 4 & 4 & \cdots
 \\ 
   \hline
 H_I(d)    \    & 3 & 4  &  
 \cdots\\ 
   \hline
 \delta_I(d) &2& 1& \cdots \\ 
\hline
 {\rm fp}_{I}(d) &0& -4&\cdots \\ 
\end{array}
\right.
\end{eqnarray*}
\end{example}

Let $C_{\mathbb{X}}(d)$ be a projective Reed-Muller-type code. If
$d\geq {\rm reg}(S/I(\mathbb{X}))$, then
$\delta_\mathbb{X}(d)=1$. The converse is not true as the next example
shows.

\begin{example}\label{reg-min-dis} Let 
$\mathbb{X}=\{[(1,1,1)],[(1,-1,0)],[(1,0,-1)],[(0,1,-1)],[(1,0,0)]\}$
and let $I$ be its vanishing ideal over the finite field
$\mathbb{F}_3$. Using {\em Macaulay\/}$2$
\cite{mac2} we obtain that ${\rm reg}(S/I)=3$. Notice that
$\delta_\mathbb{X}(1)=1$ because the polynomial $t_1+t_2+t_3$ 
vanishes at all points of $\mathbb{X}\setminus\{[(1,0,0)]\}$. 
\end{example}

The next example shows that $\delta_I$ is not in general
non-increasing. This is why we often require that the dimension of $I$
be at least $1$ or that $I$ is unmixed with at least $2$ minimal
primes.

\begin{example}\label{example-min-dis-graded} Let $I$ be the ideal of
$\mathbb{F}_5[t_1,t_2]$ generated by
$t_1^7,\,t_2^5,t_1^2t_2,\,t_1t_2^3$. 
Using Corollary~\ref{oct15-15-1}
and {\em Macaulay\/}$2$ 
\cite{mac2} we get that the regularity of $S/I$ is $7$, that
is, $H_I(d)=0$ for $d\geq 7$, and 
\begin{eqnarray*}
\hspace{-11mm}&&\left.
\begin{array}{c|c|c|c|c|c|c|c}
d & 1 & 2 & 3 &4 &5 &6 &\cdots\\
   \hline
{\rm deg}(S/I) & 13 & 13 & 13 & 13&13 &13 & \cdots
 \\ 
   \hline
 H_I(d)    \    & 2 & 3  &3 &2 &1 &1&
 \cdots\\ 
   \hline
 \delta_{I}(d) &6& 2& 1&1&2 &1 &\cdots \\ 
\end{array}
\right.
\end{eqnarray*}
\end{example}

\medskip

\noindent {\bf Acknowledgments.} We thank the referee for a careful
reading of the paper and for the improvements suggested. 

\bibliographystyle{plain}

\end{document}